\documentclass[11pt]{amsart}

\usepackage{comment}
\usepackage[rgb,dvipsnames]{xcolor}
\usepackage{stmaryrd}
\usepackage{amsfonts}
\usepackage{amssymb}
\usepackage{amsmath,amsthm}
\usepackage{latexsym}
\usepackage{amscd}
\usepackage{esint}
\usepackage{mathrsfs}
\usepackage{dsfont}
\usepackage{setspace}
\usepackage{enumitem}
\usepackage[margin=1.25in]{geometry}
\usepackage{bm}
\usepackage{epsfig}
\usepackage{subcaption}

\usepackage{hyperref}
\usepackage{comment}
\usepackage{enumitem}

\newtheorem{theorem}{Theorem}[section]
\newtheorem{corollary}[theorem]{Corollary}
\newtheorem{claim}[]{Claim}
\newtheorem{lemma}[theorem]{Lemma}

\newtheorem{proposition}[theorem]{Proposition}

\theoremstyle{definition}
\newtheorem{definition}[theorem]{Definition}
\newtheorem{conjecture}[theorem]{Conjecture}

\theoremstyle{remark}
\newtheorem{remark}[theorem]{Remark}

\numberwithin{equation}{section}

\newcommand{\V}{\mathcal{V}}
\newcommand{\RV}{\mathcal{RV}}

\newcommand{\R}{\mathbb{R}}
\newcommand{\N}{\mathbb{N}}
\newcommand{\mH}{\mathcal{H}}
\newcommand{\F}{\mathcal{F}}

\newcommand{\B}{\mathcal{B}}
\newcommand{\C}{\mathcal{C}}

\newcommand{\Scal}{\mathcal{S}}

\newcommand{\mB}{\mathbb{B}}
\newcommand{\mAn}{\mathbb{A}}
\newcommand{\mZ}{\mathbb{Z}}
\newcommand{\Z}{\mathcal{Z}}

\newcommand{\M}{\mathbf{M}}

\newcommand{\mf}{\mathbf{f}}

\newcommand{\mF}{\mathbf{F}}
\newcommand{\mI}{\mathbf{I}}
\newcommand{\mR}{\mathcal{R}}

\newcommand{\btau}{\boldsymbol \tau}
\newcommand{\bleta}{\boldsymbol \eta}

\newcommand{\tM}{\widetilde{M}}

\newcommand{\tB}{\widetilde{B}}
\newcommand{\tBcal}{\widetilde{\mathcal{B}}}

\newcommand{\Area}{\textrm{Area}}

\newcommand{\spt}{\operatorname{spt}}
\newcommand{\dist}{\operatorname{dist}}
\newcommand{\Div}{\operatorname{div}}

\newcommand{\inj}{\operatorname{inj}}

\newcommand{\interior}{\operatorname{int}}
\newcommand{\Ric}{\operatorname{Ric}}
\newcommand{\Clos}{\operatorname{Clos}}

\newcommand{\rom}[1]{\expandafter\romannumeral #1}

\newcommand{\mfX}{\mathfrak{X}}

\newcommand{\bd}{\partial}

\renewcommand{\Div}{\operatorname{div}}

\title{Generic density of equivariant min-max hypersurfaces}
\author{Tongrui Wang}
\address{School of Mathematical Sciences, Shanghai Jiao Tong University, Minhang District, Shanghai 200240, China}
\email{wangtongrui@sjtu.edu.cn}

\begin{document}

\begin{abstract}
For a compact Riemannian manifold $M^{n+1}$ acted isometrically on by a compact Lie group $G$ with cohomogeneity ${\rm Cohom}(G)\geq 2$, we show the Weyl asymptotic law for the $G$-equivariant volume spectrum. 
As an application, we show in the $C^\infty_G$-generic sense with a certain dimension assumption that the union of min-max minimal $G$-hypersurfaces (with free boundary) is dense in $M$, whose boundaries' union is also dense in $\partial M$.  
\end{abstract}

\keywords{Equivariant min-max theory, Free boundary minimal hypersurfaces, Weyl law, Equivariant volume spectrum}
\subjclass[2010]{Primary 53A10, 53C42}
\maketitle

\section{Introduction}

\subsection{Closed minimal hypersurfaces}\label{Subsec: introduction-closed}
In differential geometry, a classical and important theme is to establish a general existence theory for minimal hypersurfaces in Riemannian manifolds.
In the 1960s, Almgren \cite{almgren1962homotopy}\cite{almgren1965theory} developed a variational theory called {\em min-max theory} to construct minimal submanifolds of any codimension in compact Riemannian manifolds. 
In an $(n+1)$-dimensional closed manifold $M^{n+1}$, the regularity theory for min-max minimal hypersurfaces was later established by Pitts \cite{pitts2014existence} ($2\leq n\leq 5$) and Schoen-Simon \cite{schoen1981regularity} ($n\geq 6$). 
Based on Morse theoretic considerations, S.-T. Yau \cite{yau1982seminar} made the following conjecture in his celebrated 1982 Problem Section. 
\begin{conjecture}[\cite{yau1982seminar}]\label{Con: Yau}
	Every closed three-dimensional Riemannian manifold $(M^3, g)$ contains infinitely many (immersed) minimal surfaces. 
\end{conjecture}

In recent years, with the rich development of the min-max theory, this conjecture has been confirmed from different perspectives by a series of works. 

Specifically, Marques-Neves made the first progress of Conjecture \ref{Con: Yau} in positive Ricci curvature closed manifolds of dimensions $3\leq \dim(M)\leq 7$ based on the multi-parameter min-max theory they extended in \cite{marques2014min}. 
Another key ingredient in \cite{marques2017existence} is the non-linear growth rate of the {\em volume spectrum} $\{\omega_{p}(M^{n+1}, g_{_M})\}_{p\in\mZ_+}$, which was first proposed and studied by Gromov \cite{gromov1988dimension}\cite{gromov2003isoperimetry}\cite{gromov2009SingularitiesEA}\cite{gromov2015morse} and later Guth \cite{guth2009minimax} as the analog of the spectrum for Laplacian. 
Subsequently, Marques-Neves pioneered a Morse theory for the area functional by demonstrating in \cite{marques2016morse} that the $p$-th volume spectrum $\omega_p(M,g_{_M})$ can be realized by $\sum_{i=1}^{N_p}m_i^{(p)}\Area(\Sigma_i^{(p)})$ for some positive integers $\{m_i^{(p)}\}_{i=1}^{N_p}$ and disjoint embedded closed minimal hypersurfaces $\{\Sigma_{i}^{(p)}\}_{i=1}^{N_p}$ in $M$ with $\sum_{i=1}^{N_p}{\rm index}(\Sigma_i^{(p)})\leq p$. 
In addition, they further proposed the Multiplicity One Conjecture (i.e. $m_i^{(p)}\equiv 1$ for generic metrics), which was first studied by Zhou \cite{zhou2015min}\cite{zhou2017min} in manifolds with positive Ricci curvature and confirmed later in \cite{zhou2020multiplicity}. 
Combining Zhou's resolution of the multiplicity one conjecture with Marques-Neves' program on Morse index lower bounds \cite{marques2021morse}, it was concluded under generic sense that there is a closed minimal hypersurface of Morse index $p$ and area $\omega_p(M,g_{_{M}})$ for each $p\in\mZ_+$, which resolved Yau's conjecture in the generic sense from the perspective of Morse theory. 

Starting from another perspective, there are also some significant advancements in the spatial distribution of closed minimal hypersurfaces revealed by min-max theory coupled with a Weyl's asymptotic law \cite{liokumovich2018weyl} (conjectured by Gromov \cite[8.4]{gromov2003isoperimetry}) for the volume spectrum $\{\omega_{p}(M^{n+1}, g_{_M})\}_{p\in\mZ_+}$, i.e. 
\begin{align}\label{Eq: classical weyl law}
	\lim _{p \rightarrow \infty} \omega_{p}(M^{n+1}, g_{_M}) p^{-\frac{1}{n+1}}=a(n) \operatorname{vol}(M, g_{_M})^{\frac{n}{n+1}}.
\end{align}
Among these, Irie-Marques-Neves \cite{irie2018density} firstly showed the generic density of minimal hypersurfaces as a stronger generic version of Yau's conjecture in closed manifolds $M^{n+1}$ with $3\leq n+1\leq 7$. 
Later on, Marques-Neves-Song \cite{marques2019equidistribution} quantified this result by showing the generic existence of a sequence of closed minimal hypersurfaces equidistributed in $M$. 
One can also refer to \cite{li2023existence} for the generic infinite existence result in higher dimensional manifolds. 
Moreover, Song-Zhou \cite{song2020generic} also obtained a generic scarring result concerning a sequence of closed minimal hypersurfaces with large area and index quantitatively surrounding any stable minimal hypersurface. 

Finally, for general Riemannian metrics, Yau's conjecture was eventually confirmed by Song's remarkable work \cite{song2018existence} for closed manifolds of dimensions between $3$ and $7$.

\subsection{Free boundary minimal hypersurfaces}
In compact manifolds $M^{n+1}$ with non-empty smooth boundary $\bd M\neq\emptyset$ and dimension $3\leq n+1\leq 7$, most of the corresponding conclusions mentioned above have also been proved to hold for {\em free boundary minimal hypersurfaces} (abbreviated as FBMHs). 

For instance, Li-Zhou \cite{li2021min} established the first full general version of free boundary min-max theory, which further led to the infinite existence of FBMHs in non-negative Ricci curvature compact manifolds with strictly convex boundary parallel to the result in \cite{marques2017existence}. 
One can also refer to \cite{guang2021min} and \cite{sun2020multiplicity} for the index upper bounds and generic multiplicity one for min-max FBMHs. 
Subsequently, based on the free boundary min-max theory of Li-Zhou \cite{li2021min}, the density of FBMHs in $M$ and the density of FBMHs' boundaries in $\bd M$ were accomplished in \cite{guang2021min} and \cite{wang2022existence} respectively under a generic choice of Riemannian metrics. 
Furthermore, there is also a recent progress made by Wang \cite{wang2020existence} which resolved Yau's conjecture in the free boundary setting with full generality.

\subsection{Equivariant minimal hypersurfaces}
Motivated by these excellent results based on Almgren-Pitts' min-max, the author generalized and established the multi-parameter Almgren-Pitts min-max theory under certain symmetry constraints through a series of works (\cite{wang2022min}\cite{wang2023min}\cite{wang2023equivariant}). 
To be specific, consider a compact Riemannian manifold $(M^{n+1},g_{_M})$ and a compact Lie group $G$ acting as isometries on $M$ so that either 
\begin{align}
	\tag{\dag} &\mbox{$\bd M = \emptyset$ and $3\leq {\rm codim}(G\cdot x)\leq 7$ for all $x\in M$; or} \label{Eq: assumption closed} \\
	\tag{\ddag} & \mbox{$\bd M \neq \emptyset$, $3\leq n+1 \leq 7$, and ${\rm Cohom}(G):= \min_{x\in M} {\rm codim}(G\cdot x) \geq 3$.} \label{Eq: assumption with boundary}
\end{align}
Then the equivariant min-max constructions in closed (\cite{wang2022min}\cite{wang2023equivariant}) and free boundary (\cite{wang2023min}) scenarios confirmed the existence of a (closed or free boundary) $G$-invariant minimal hypersurface in $M$. 
(Note there is an additional assumption on the non-principal orbits union $M\setminus M^{prin}$ in \cite{wang2022min} for technical reasons, which was later removed by the constructions in \cite[Appendix C]{wang2023min} (see \cite[Section 4]{wang2023equivariant} for the modified version).)
One can also refer to \cite{ketover2016equivariant} and \cite{liu2021existence} for the existence result by other settings of equivariant min-max with some additional assumptions on group actions. 
Therefore, it now seems plausible to search for the equivariant generalization of Yau's conjecture by the equivariant min-max approach. 

By characterizing the genus of minimal surfaces, Ketover \cite{ketover2016free} found a family of dihedral symmetric minimal surfaces in the 3-dimensional Euclidean ball with three boundary connected components whose genus tends to infinity. 
In addition to the specific ambient spaces and symmetries, it was proved by the author in \cite{wang2022min}\cite{wang2023min} that there are infinitely many closed (resp. free boundary) $G$-invariant minimal hypersurfaces when $M$ has positive Ricci curvature (resp. $\Ric_M\geq 0$ and strictly convex boundary $\bd M$) so that (\ref{Eq: assumption closed}) (resp. (\ref{Eq: assumption with boundary})) is satisfied. 
Moreover, in the completely general case, the existence of at least $(l+1)$ disjoint closed (resp. free boundary) $G$-invariant minimal hypersurfaces was also obtained in \cite{wang2022min} (resp. \cite{wang2023min}), where $l+1:=\min_{x\in M}{\rm codim}(G\cdot x)$ can be seen as the dimension of orbit space $M/G$. 

In this paper, we will show in $C^\infty_G$-generic sense that there are infinitely many (closed or free boundary) $G$-invariant minimal hypersurfaces under the dimension assumptions (\ref{Eq: assumption closed}) or (\ref{Eq: assumption with boundary}). 
Indeed, we shall show a stronger result as in \cite{irie2018density}\cite{wang2022existence} that is the $C^\infty_G$-generic density of $G$-invariant FBMHs in $M$ and the $C^\infty_G$-generic density of those hypersurfaces' boundaries in $\bd M$. 
Specifically, we have the following main theorem generalizing the results in \cite{irie2018density}\cite{guang2021min}\cite{wang2022existence}. 

\begin{theorem}\label{Thm: main density}
	Let $M^{n+1}$ be an $(n+1)$-dimensional compact manifold with (possibly empty) boundary $\bd M$ and $G$ be a compact Lie group acting by diffeomorphisms on $M$. 
	Suppose either (\ref{Eq: assumption closed}) or (\ref{Eq: assumption with boundary}) is satisfied.
	Then for a generic smooth $G$-invariant Riemannian metric $g_{_M}$ on $M$, we have 
	\begin{itemize}
		\item[(i)] the union of $G$-invariant FBMHs is dense in $M$;
		\item[(ii)] the union of the boundaries of $G$-invariant FBMHs is also dense in $\bd M$. 
	\end{itemize}
\end{theorem}

Note that in the above theorem and the subsequent content, it is convenient and unambiguous to also regard closed minimal hypersurfaces as FBMHs (with empty boundary). 

\begin{remark}\label{Rem: main density}
	We mention that the above density results are also valid if one only considers the union of {\em min-max $G$-hypersurfaces} (Definition \ref{Def: min-max $(c,G)$-hypersurfaces}) realizing the equivariant volume spectrum (Definition \ref{Def: p-width}) given in Theorem \ref{Thm: Gp-width realization}.
\end{remark}

It should be noted that the dimension assumptions (\ref{Eq: assumption closed})(\ref{Eq: assumption with boundary}) are used to guarantee the regularity theory in equivariant min-max constructions (cf. \cite{wang2022min}\cite{wang2023min}\cite{wang2023equivariant}). 
If $\min_{x\in M} {\rm codim}(G\cdot x)=2$, one would get singularities in our equivariant min-max hypersurfaces similar to Almgren-Pitts' min-max for geodesic curves. 
Nevertheless, one can still expect in this case a similar density result for certain singular $G$-invariant minimal hypersurfaces in light of \cite{liokumovich2023generic}. 
Meanwhile, for the case of $\min_{x\in M} {\rm codim}(G\cdot x)=1$, all the minimal $G$-hypersurfaces are area critical orbits, and the min-max theory is not applicable as $\dim(M/G)=1$. 
It should also be noted that there is no $G$-invariant hypersurface when $\min_{x\in M} {\rm codim}(G\cdot x)=0$. 

Additionally, it seems FBMHs still lack an optimal regularity theory similar to the result in \cite{schoen1981regularity} (to the author's best knowledge). Hence, a stronger assumption $3\leq n+1\leq 7$ was used in (\ref{Eq: assumption with boundary}) compared to the assumption $3\leq {\rm codim}(G\cdot x)\leq 7$ in (\ref{Eq: assumption closed}). 
We expect this difficulty can be resolved in the future. 

In particular, using \cite{schoen1981regularity}, if one allows $G$-invariant minimal hypersurfaces possess singular sets with Hausdorff dimension no more than $n-7$, then we also have the following $C^\infty_G$-generic existence of infinitely many {\em closed} $G$-invariant minimal hypersurfaces in arbitrary dimensional closed $G$-manifolds satisfying $3\leq {\rm codim}(G\cdot x)$ for all $x\in M$. 
\begin{theorem}\label{Thm: main infinitely many singular MHs}
	Let $M^{n+1}$ be an $(n+1)$-dimensional {\em closed} manifold and $G$ be a compact Lie group acting by diffeomorphisms on $M$ so that $3\leq {\rm cohom}(G):=\min_{x\in M} {\rm codim}(G\cdot x) $. 
	Then for a $C^\infty_G$-generic Riemannian metric $g_{_M}$ on $M$, there exist infinitely many {\em closed} $G$-invariant minimal hypersurfaces whose singular sets have Hausdorff dimension no more than $n-7$. 
\end{theorem}

\subsection{Key ingredients}
The above results are mainly based on an equivariant version of the perturbation arguments in \cite{irie2018density}\cite{li2023existence}\cite{wang2022existence} using the equivariant min-max theory. 
The key novelties include a compactness theorem for min-max $G$-hypersurfaces with free boundary and a Weyl's asymptotic law for the equivariant volume spectrum. 

To be specific, consider a compact Riemannian manifold $(M^{n+1},g_{_M})$ and a compact Lie group $G$ acts as isometries so that either (\ref{Eq: assumption closed}) or (\ref{Eq: assumption with boundary}) is satisfied.  
Intuitively, for any $p\in\mZ_+$, we collect all the non-trivial $p$-dimensional complexes in the $G$-invariant relative $n$-cycles space $\Z_n^G(M,\bd M;\mZ_2)$, which is known as the set of {\em $(G,p)$-sweepouts} of $M$ (Definition \ref{Def: p-sweepouts}) and denoted by $\mathcal{P}^G_p(M)$. 
Then the {\em $(G,p)$-width} of $M$ (also called the {\em $p$-th $G$-equivariant volume spectrum}) is defined by the following min-max value of mass
\[ \omega^G_p(M, g_{_M}) := \inf_{\Phi\in \mathcal{P}_p^G(M)} \sup_{x\in {\rm dmn}(\Phi)} \M(\Phi(x)). \]
In particular, if $G=\{id\}$, this concept of $\{\omega_p^G(M)\}_{p\in\mZ_+}$ will reduce to the classical volume spectrum $\{\omega_p(M)\}_{p\in\mZ_+}$ of $M$ as mentioned in Section \ref{Subsec: introduction-closed}. 
Then we notice the perturbation arguments in \cite{irie2018density}\cite{li2023existence}\cite{wang2022existence} relied heavily on the following facts:
\begin{itemize}
	\item[(1)] $\omega_p(M)$ can be realized by the area of some disjoint minimal hypersurfaces with multiplicities (and free boundary) due to certain compactness theorems (e.g. \cite{sharp2017compactness}\cite{ambrozio2018compactness});
	\item[(2)] the volume spectrum $\{\omega_p(M)\}_{p\in\mZ_+}$ satisfies the Weyl's asymptotic law (\ref{Eq: classical weyl law}).
\end{itemize}
Thus, we shall generalize the above facts to the equivariant case as our key ingredients.  

Firstly, we mention that the equivariant version of (1) has been built in the closed case (\ref{Eq: assumption closed}). 
Indeed, in \cite{wang2023equivariant}, the author built a compactness theorem (\cite[Theorem 1.5]{wang2023equivariant}) for closed minimal $G$-hypersurfaces constructed via equivariant min-max, which was further applied to show $\sum_{i=1}^{N_p}m_i^{(p)}\Area(\Sigma_i^{(p)})=\omega_p^G(M, g_{_M})$ for some disjoint min-max minimal $G$-hypersurfaces $\{\Sigma_{i}^{(p)}\}_{i=1}^{N_p}$ with integer multiplicities $\{m_i^{(p)}\}_{i=1}^{N_p}$ (\cite[Corollary 1.6]{wang2023equivariant}). 
In this paper, we will show the counterparts (Theorem \ref{Thm: compactness theorem}, \ref{Thm: Gp-width realization}) of these results in the free boundary case provided the assumption (\ref{Eq: assumption with boundary}). 

Note \cite[Corollary 1.7]{wang2023equivariant} also indicates the $G$-index of $\cup_{i=1}^{N_p}\Sigma_i^{(p)}$ is bounded from above by $p$ in the closed case (\ref{Eq: assumption closed}). 
However, since this result is not required for our purpose, we leave its free boundary generalization to other places for the sake of conciseness.


Next, consider the equivariant version of (2). 
By generalizing the bend-and-cancel technique of Guth \cite{guth2009minimax}, the author has shown in \cite{wang2022min} that the $(G,p)$-width has a non-linear growth rate determined by the dimension of the orbit space, i.e. 
\[ C_1 p^{\frac{1}{l+1}} \leq \omega^G_p(M, g_{_M})  \leq C_2 p^{\frac{1}{l+1}} \qquad\mbox{for all $p\in \mZ_+$},\]
where $l+1 :=\min_{x\in M}{\rm codim}(G\cdot x)$ and $C_1,C_2>0$ are constants depending only on $(M^{n+1},g_{_M})$ and $G$. 
Regarding $\{\omega_p^G(M, g_{_M}) \}_{p=1}^\infty$ as the volume spectrum of the singular orbit space $M/G$ with respect to a weighted area functional, we further set out the following Weyl's asymptotic law as an equivariant generalization of Gromov's conjecture (\cite[\S 8]{gromov2003isoperimetry}). 
\begin{theorem}[Weyl Law for the Equivariant Volume Spectrum]\label{Thm: main theorem weyl law}
	Suppose $(M^{n+1}, g_{_M})$ is a connected compact Riemannian manifold with smooth (possibly empty) boundary, and $G$ is a compact Lie group acting by isometries on $M$ so that $\min_{x\in M}{\rm codim}(G\cdot x) = l+1 \geq 2$. 
	Let $M^{prin}$ be the union of all principal orbits in $M$ and $g_{_{M/G}}$ be the induced metric on $M^{prin}/G$ so that $\pi: M^{prin}\to M^{prin}/G$ is a Riemannian submersion. 
	Define $\tilde{g}_{_{M/G}}([p]) := \big(\mH^{n-l}(\pi^{-1}([p]))\big)^{\frac{2}{l}}\cdot g_{_{M/G}}([p])$ as a weighted metric. 
	Then,
	\begin{eqnarray*}
		\lim_{p\to\infty}p^{-\frac{1}{l+1}}\omega_p^G(M, g_{_M}) &=&  a(l){\rm Vol}(M^{prin}/G, \tilde{g}_{_{M/G}})^{\frac{l}{l+1}}
		\\
		&=& a(l)\cdot \Big( \int_{M} (\mH^{n-l}(G\cdot q))^{\frac{1}{l}}d\mH^{n+1}(q) \Big)^{\frac{l}{l+1}},
	\end{eqnarray*}
	where $a(l)$ is a constant depending only on $l$ given by \cite[Theorem 1.1]{liokumovich2018weyl}. 
\end{theorem}
\begin{remark}
	By analogy with the equidistribution result in \cite{marques2019equidistribution}, we also conjecture the $C^\infty_G$-generic existence of a sequence of $G$-invariant minimal hypersurfaces whose projections are equidistributed in $M/G$ under the weighted metric $\tilde{g}_{_{M/G}}$. 
\end{remark}

The proof of Theorem 1.5 can be outlined as follows. 
Firstly, if $M$ contains only a single orbit type, i.e. $M=M^{prin}$. (E.g. $G$ acts freely on $M$.) 
Then the quotient map $\pi: (M,g_{_M})\to (M/G, g_{_{M/G}})$ is a Riemannian submersion so that $\Area_{g_{_M}}(\Sigma) = \Area_{\tilde{g}_{_{M/G}}}(\pi(\Sigma))$ for all $G$-invariant hypersurface $\Sigma$ by the co-area formula (cf. \cite[Theorem 2]{hsiang1971minimal}). 
Intuitively, we shall have $\omega_p^G(M, g_{_M}) =  \omega_p(M/G, \tilde{g}_{_{MG}})$ in this case, which directly indicates Theorem \ref{Thm: main theorem weyl law} by \cite[Theorem 1.1]{liokumovich2018weyl}. 

For general $G$-actions, we shall make use of the openness and density of $M^{prin}$ to cut off a small $G$-invariant neighborhood of $M\setminus M^{prin}$ and obtain a compact $G$-invariant domain $\Omega$ in $M^{prin}$ with piecewise smooth boundary. 
Note the boundary of the tubular neighborhood has a finite area depending on $\mH^n(M\setminus M^{prin})<\infty$, which is negligible compared to $p^{\frac{1}{l+1}}$. 
Hence, by constructing a $G$-equivariant map $F$ inflating $\Omega$ to the whole manifold $M$, we can associate any $(G,p)$-sweepout $\Phi $ in $\Omega$ to a $(G,p)$-sweepout $\Psi $ in $M$ so that 
$\M(\Psi(x)) \leq {\rm Lip}(F)^{n} \cdot \big[ \M(\Phi(x)) + (1+2\mH^n(M\setminus M^{prin} ) + \mH^n(\partial M) ) \big]$. 
For $M\setminus\Omega$ sufficiently small, ${\rm Lip}(F)$ will be sufficiently close to $1$, which gives the upper limit $\limsup_{p\to\infty}p^{-\frac{1}{l+1}} \omega^G_p(M,g_{_M})$ as the right-hand side in Theorem \ref{Thm: main theorem weyl law}. 
Meanwhile, the lower limit $\liminf_{p\to\infty}p^{-\frac{1}{l+1}} \omega^G_p(M,g_{_M})$ follows similarly to \cite[Theorem 3.1]{liokumovich2018weyl} by the Lusternik-Schnirelmann theory since $\Omega \subset\subset M$.

Finally, combining 
the above key ingredients with the perturbation arguments in \cite{irie2018density}\cite{li2023existence}\cite{wang2022existence}, we conclude the $C^\infty_G$-generic density of $G$-invariant FBMHs (and their boundaries) in manifolds with group actions satisfying (\ref{Eq: assumption closed}) or (\ref{Eq: assumption with boundary}) as the main applications Theorem \ref{Thm: main density}, \ref{Thm: main infinitely many singular MHs}.

\subsection{Outline}
In Section \ref{Sec: preliminary}, we collect some notations in Lie group actions, geometric measure theory, and free boundary minimal hypersurfaces. 
We also introduce the equivalent classes of relative cycles in the formulation of integer rectifiable currents as well as the formulation of integral currents, which are equivalent by Proposition \ref{Prop: equivalence between two formulation}. 
In Section \ref{Sec: min-max setting}, we first introduce the isomorphism between $\Z_n^G(M,\partial M;\mZ_2)$ and $H_{n+1}(M, \partial M;\mZ_2)$. 
Then we generalize the volume spectrum into the equivariant case. 
Using the compactness argument in \cite{wang2023equivariant}, we further show the equivariant volume spectrum can be realized by the area of some min-max $G$-hypersurfaces with free boundary. 
After that, a Weyl's asymptotic behavior (Theorem \ref{Thm: main theorem weyl law}) is proved in Section \ref{Sec: weyl law}. 
Finally, we show the $C^\infty_G$-generic density of $G$-invariant FBMHs and their boundaries by a perturbation argument in Section \ref{Sec: denseness of FBMHs}.

{\bf Acknowledgement.} The author would like to thank Prof. Gang Tian for his constant encouragement. 
	Part of this work was done during the author's visit to Prof. Xin Zhou at Cornell University; he is grateful for their hospitality.
	The author also thanked Zhiang Wu and Yangyang Li for helpful discussions, and also thanked the anonymous referees and reviewers for helpful comments. 
	The author is partially supported by China Postdoctoral Science Foundation 2022M722844.

\section{Preliminary}\label{Sec: preliminary}
 
In this paper, we consider a connected compact Riemannian $(n+1)$-dimensional manifold $(M^{n+1}, g_{_M})$ with smooth (possibly empty) boundary $\partial M$. 
Let $G$ be a compact Lie group acting as isometries on $M$. 
Then we denote by $M/G$ the orbit space of $M$, which is a Hausdorff metric space with induced distance $\dist_{M/G} ([p], [q]) := \dist_M(G\cdot p, G\cdot q) $. 
We also denote by $\mu$ the bi-invariant Haar measure on $G$ which has been normalized to $\mu(G)=1$.

To begin with, we collect some definitions for group actions. 
Given $p\in M$, let $G_p:=\{g\in G: g\cdot p = p\}$ be the isotropy group of $p$, which is a Lie subgroup of $G$. 
One easily verifies that $g\cdot G_p \cdot g^{-1} = G_{g\cdot p}$. 
Given any Lie subgroup $H$ of $G$, denote by $(H)$ the conjugate class of $H$ in $G$.  
Then we say an orbit $G\cdot p$ has the {\em $(H)$ orbit type} if $G_p\in (H)$. 
Define the {\em $(H)$ orbit type stratum} $M_{(H)}$ by the union of orbits with $(H)$ orbit type:
\[M_{(H)}:=\{p\in M: G_p\in (H)\}.\] 
It follows from the compactness of $M$ that there are only a finite number of different orbit types in $M$ (\cite[Proposition 2.2.5]{berndt2016submanifolds}). 
In addition, we also have a partial order on the set of orbit types. 
Namely, for any subgroups $H$ and $K$ of $G$, we write $(H)\leq (K)$ if $H$ is conjugate to a subgroup of $K$. 
It is also well known that $M_{(H)}$ is a disjoint union of smooth embedded submanifolds of $M$, and $\Clos(M_{(H)})  \subset \cup_{(H)\leq (G_x)} M_{(G_x)}$ (see \cite[Chapter IV, Theorem 3.3]{bredon1972introduction}). 
The decomposition by orbit types gives a {\em stratification} structure of $M$ as well as $M/G$. 

Moreover, in the connected manifold $M$, there is a unique minimal orbit type $(P)$, i.e. $(P)\leq (G_p)$ for all $p\in M$, which is known as the {\em principal orbit type}. 
Denote by 
\[M^{prin}:=M_{(P)}\subset M\]
the union of all principal orbits, which forms an open dense $G$-invariant submanifold of $M$. 
Then, we define the {\em cohomogeneity} ${\rm Cohom}(G)$ of $G$ by the co-dimension of a principal orbit, i.e. ${\rm Cohom}(G)=\min_{x\in M}{\rm codim}(G\cdot x)$. 
Unless otherwise specified, we always assume that the cohomogeneity ${\rm Cohom}(G)$ is at least $2$. 

\subsection{Notations in manifolds and Euclidean spaces}\label{Subsec: notations in manifolds}
Next, we introduce some notations in manifolds and Euclidean spaces. 

By an equivariant extending argument (cf. \cite[Appendix A]{wang2023min}), there exists a connected $(n+1)$-dimensional closed Riemannian $G$-manifold $(\widetilde{M}, g_{_{\widetilde{M}}})$ with $G$ acts by isometries so that $M$ is equivariantly isometrically embedded in $\widetilde{M}$. (Indeed, $M$ is isometric to a branch of $\widetilde{M}$). 
Therefore, we always regard $M$ as a compact domain with smooth boundary in a closed Riemannian $G$-manifold $(\widetilde{M}, g_{_{\widetilde{M}}})$. 

In addition, by the main result in \cite{moore1980equivariant}, there exists an orthogonal representation $\rho:G\to O(L)$ and a $G$-equivariant isometric embedding from $\widetilde{M}$ into $\R^L$, for some $L\in\N$. 
Hence, we can always consider $\widetilde{M}$, as well as $M$, to be a manifold in $\R^L$ with induced Riemannian metric and induced group $G$ actions $g\cdot x := \rho(g)(x)$ for all $ g\in G, x\in \R^L$. 

Now, for any $0<s<t<r$, we use the following notations:
\begin{itemize}
	\item $\pi$: the natural projection $\pi:\widetilde{M}\mapsto \widetilde{M}/G$ defined by $G\cdot p \mapsto [p]$;
	\item $\Clos(A)$: the closure of a subset $A$; 
	\item $\tB_r(A)$: the open geodesic $r$-neighborhood of $A\subset\widetilde{M}$;
	\item $B_r(A)$: the open $r$-neighborhood of $A\subset\widetilde{M}/G$ under the distance $\dist_{\widetilde{M}/G}$;
	\item $\mB_r(A)$: the open Euclidean $r$-neighborhood of $A\subset\R^L$;
	\item $\mAn_{s,t}(A)$: the open Euclidean annulus around $A\subset\R^L$ given by $\mB_t(A)\setminus\mB_s(A) $;
\end{itemize}
A subset $A\subset \widetilde{M}$ or $A\subset M$ is said to be {\em $G$-invariant}, if $A=G\cdot A := \cup_{p\in A}G\cdot p$. 
Since $G$ acts by isometrics, we see $\tB_r(A)$, $\mB_r(A)$ and $\mAn_{s,t}(A)$ are all $G$-invariant open sets provided $A$ is $G$-invariant. 
In addition, if $A\subset \widetilde{M}$ is a $G$-invariant submanifold of $\tM$ (abbreviated as $G$-submanifold), then we denote by 
\[\exp_{A}^\perp : {\bf N}A \to \tM \]
the normal exponential map of $A\subset \tM$, where ${\bf N}A$ is the normal bundle over $A$ in $\tM$. 
Moreover, for $G$-submanifold $A\subset \tM$, we can define the $G$-actions on ${\bf N}A$ by the tangent map $g\cdot v := dg(v)$ for all $g\in G, v\in {\bf N}A$. 
Hence, $\exp_{A}^\perp$ gives a $G$-equivariant diffeomorphism in a small $G$-invariant {\em tubular neighborhood} $\widetilde{T}_{\inj(A)}(A)$, i.e. $\exp_{A}^\perp(dg(v)) = g\cdot \exp_{A}^\perp(v)$ for $\exp_{A}^\perp(v)\in \widetilde{T}_{\inj(A)}(A)$, where $\inj(A)$ stands for the injectivity radius of $A\subset \tM$. 
In particular, $\widetilde{T}_{r}(A) = \tB_r(A)$ if $A$ is a closed submanifold and $r\in (0,\inj(A))$. 

In particular, if we take the closed $G$-submanifold as an orbit $G\cdot p\subset \tM$, then we also denote by $\tB_r^G(p)$ the open geodesic tube of radius $r$ centered at $G\cdot p$.

\subsection{Geometric measure theory}

In this subsection, we introduce some notations in geometric measure theory, which are referenced from \cite{pitts2014existence} and \cite{simon1983lectures}. 

Let $\RV_k (M)$ be the space of rectifiable $k$-varifolds in $\R^L$ supported in $M$. 
Then we define $\V_k(M)$ to be the closure of $\RV_k (M)$ in the weak topology, and define the $\mF$-metric on $\V_k(M)$ as in \cite[2.1(19)]{pitts2014existence}, which induces the weak topology on any mass bounded subset of $\V_k(M)$. 
For any $V\in \V_k(M)$, denote by $\|V\|$ the weight measure on $M$ induced by $V$.
In addition, if $g_\#V=V$ for all $g\in G$, then we say $V\in \V_k(M)$ is $G$-invariant, and denote by $\V_k^G(M)$ the subspace of $\V_k(M)$ containing all the $G$-invariant $k$-varifolds in $M$.

Let $\mR_k(M; \mZ_2)$ (resp. $\mR_k(\partial M; \mZ_2)$) be the space of $\mZ_2$-coefficients rectifiable $k$-currents in $\R^L$ supported in $M$ (resp. $\partial M$). 
For any $T\in \mR_k(M; \mZ_2)$, denote by $|T|$ and $\|T\|$ the integral varifold and the Radon measure induced by the representative modulo $2$ of $T$ (see \cite[Page 430]{federer2014geometric}). 
Additionally, the mass norm and the flat semi-norm on $\mR_k(M; \mZ_2)$ are denoted by $\M$ and $\F=\F^M$ (see \cite[4.2.26]{federer2014geometric}). 

Let $\C(M)$ be the collection of all Caccioppoli sets in $M$, i.e. the sets with finite perimeter. 
Then, consider the space of $k$-dimensional relative $\mZ_2$-cycles in $(M,\bd M)$
\[Z_{k}(M, \partial M; \mZ_2) := \{ T\in \mR_k (M;\mZ_2) : \spt(\partial T)\subset \partial M \}.\] 
Define the equivalent relation $\sim$ in $Z_{k}(M, \partial M; \mZ_2)$ (see \cite{li2021min}) by
\begin{eqnarray*}
	T\sim S ~~\Leftrightarrow ~~ T-S\in \mR_k(\partial M; \mZ_2)
\end{eqnarray*}
for all $T,S\in Z_{k}(M, \partial M; \mZ_2)$.
Let $[T]$ denote the equivalent class of $T\in Z_{k}(M, \partial M; \mZ_2)$ for the relation $\sim$. 
Then for any such equivalent class $\tau = [T]$, the mass norm $\M$, the flat norm $\F$ and the support of $\tau$ are given as in \cite{li2021min} by 
\[\M(\tau):= \inf \{\M(S):~ S\in \tau \} , \quad \F(\tau):=\inf\{\F(S):~ S\in\tau\}, \quad \spt(\tau) := \cap_{S\in\tau} \spt(S) .\]
Now, we define the space of equivalence classes of {\em boundary type} relative $\mZ_2$-cycles as  
\[ \Z_n(M,\bd M;\mZ_2) := \left\{ [\bd \Omega] : \Omega\in \C(M) \right\} ,\]
which is the connected component of $0$ in $Z_n(M,\bd M;\mZ_2)/\sim$ under the $\F$-topology.

One benefit of this formulation is that for any $\tau\in \Z_n( M, \partial M; \mZ_2)$ we can find a unique $T\in \tau$ with $T\llcorner \partial M = 0$, which is known as the {\em canonical representative} of $\tau$. 
Indeed, the canonical representative of $[\bd \Omega]$ with $\Omega\in\C(M)$ is given by $T:=\bd\Omega \llcorner (M\setminus\bd M)$, which satisfies $\M(T) = \M(\tau)$ and $\spt(T) = \spt(\tau)$ (see \cite{li2021min}). 
Moreover, for $\tau,\sigma \in Z_n(M,\bd M;\mZ_2)/\sim$, we also define the $\mF$-distance by 
\[ \mF(\tau,\sigma) := \F(\tau -\sigma) + \mF(|T|,|S|),\]
where $T\in\tau$ and $S\in\sigma$ are the canonical representatives. 

Next, for any $g\in G$, since $g:M\to M$ is an isometry, we can consider the push forward of currents by $g$, and say that a current $T$ is $G$-invariant if $g_\#T=T$, $\forall g\in G$. 
Therefore, we define the following spaces of $G$-invariant elements:
\begin{itemize}
	\item $\mR_k^G(M; \mZ_2) := \left\{ T\in \mR_k(M; \mZ_2) :~ g_\#T = T, ~\forall g\in G \right\}$;
	\item $\C^G(M) := \left\{\Omega\in \C(M) : ~g\cdot \Omega= \Omega,~\forall g\in G \right\}$;
	\item $Z_{n}^G(M, \partial M; \mZ_2) := \left\{ T\in Z_{n}(M, \partial M; \mZ_2) : ~ g_\#T = T, ~\forall g\in G \right\} $;
	\item $\Z_n^G(M,\bd M;\mZ_2) := \left\{ [\partial \Omega ] :~ \Omega\in \C^G(M)\right\}$. 
\end{itemize}
We mention that $\Z_n^G(M,\bd M;\mZ_2) \subsetneq \{ \tau\in \Z_n(M,\bd M; \mZ_2) : g_\#\tau=\tau, \forall g\in G \}$ in general. 
Additionally, the mass norm $\M$ and the flat semi-norm $\F$ are naturally induced to these subspaces of $G$-currents and $G$-equivalence classes. 
Indeed, since $G$ acts by isometries, we have $\M(T) = \M(g_\#T)$, $\F(T) = \F(g_\#T)$ for any $T\in \mR_k(M; \mZ_2)$, and thus these subspaces of $G$-invariant elements are closed subspaces.

In \cite{wang2023min}, the author used the above formulation to demonstrate the equivariant free boundary min-max constructions. 
Meanwhile, another formulation using integral currents is used in \cite{almgren1962homotopy}\cite{liokumovich2018weyl}. 
Specifically, let $\mI_k(M,\mZ_2):=\{T\in \mR_k(M;\mZ_2): \bd T\in \mR_{k-1}(M;\mZ_2)\}$ be the space of $k$-dimensional integral $\mZ_2$-flat chains. 
Define
\begin{eqnarray*}
	Z_{n,rel}(M, \partial M; \mZ_2) &:=& \{ T\in \mI_n (M;\mZ_2) :~ \spt(\partial T)\subset \partial M \} ;
	\\
	T\sim_{rel} S &\Leftrightarrow &T-S\in \mI_n(\partial M; \mZ_2),
\end{eqnarray*}
for any $T,S\in Z_{n,rel}(M, \partial M; \mZ_2)$.
Then we denote by $[T]_{rel}$ the equivalent class of $T\in  Z_{n,rel}(M, \partial M; \mZ_2)$ for the relation $\sim_{rel}$, and denote by 
\begin{itemize}
	\item $\Z_{n,rel}( M, \partial M; \mZ_2) := \{[\bd \Omega]_{rel} : ~ \Omega\in \C(M) \}$;
	\item $\mI_k^G(M;\mZ_2) := \{ T\in \mI_k(M;\mZ_2) :  ~ g_\# T=T, ~\forall g\in G\}$;
	\item $Z_{n,rel}^G(M, \partial M; \mZ_2) := \{ T\in Z_{n,rel}(M, \partial M; \mZ_2) :~ g_\# T=T, ~\forall g\in G \}$;
	\item $\Z_{n,rel}^G( M, \partial M; \mZ_2) := \{[\bd \Omega]_{rel} : ~ \Omega\in \C^G(M) \}$.
\end{itemize}
Similarly, the mass norm and the flat norm for any such equivalent class $\tau = [T]_{rel}$ are defined by $\M(\tau):= \inf \{\M(S):S\in \tau \} $ and $ \F(\tau):=\inf\{\F(S): S\in\tau\} $ respectively. 
In addition, $[T]_{rel}\subset [T]$ for all $T\in Z_{n,rel}(M, \partial M; \mZ_2)$, and thus 
\begin{equation}\label{Eq: norm in two formulation}
	\M([T]) \leq \M([T]_{rel}), \qquad \F([T]) \leq \F([T]_{rel}).
\end{equation}
On the other hand, by the slicing trick in \cite[Lemma 3.8]{li2021min}, we have the following lemma.
\begin{lemma}\label{Lem: equivalence in two formulation}{\rm (\cite[Lemma 3.8]{li2021min}\cite[Lemma 3.6]{wang2023min})} 
	For any $T\in Z_n(M,\partial M;\mZ_2 )$, there exists a sequence $\{T_i\}_{i\in\N}\subset Z_{n,rel}(M, \partial M; \mZ_2)$ so that for each $i\in\N$, we have $T\sim T_i$ and $\lim_{i\to\infty} \M(T_i - T) = 0$. 
	Moreover, if $T\in Z_n^G(M,\partial M;\mZ_2 )$ is $G$-invariant, then we can also choose $\{T_i\}_{i\in\N}\subset Z_{n,rel}^G(M, \partial M; \mZ_2)$ to be $G$-invariant.
\end{lemma}

Similar to \cite[Proposition 3.2]{guang2021min}, the following proposition indicates the equivalence between these two formulations.
\begin{proposition}\label{Prop: equivalence between two formulation}
	$\Z_n^G( M, \partial M; \mZ_2)$ and $\Z_{n,rel}^G( M, \partial M; \mZ_2)$ are isometric with respect to the mass norm $\M$ and the flat norm $\F$. 
\end{proposition}
\begin{proof}
	The proof is similar to \cite[Proposition 3.2]{guang2021min}. 
	Firstly, for any $\tau\in \Z_{n,rel}^G( M, \partial M; \mZ_2)$ and $T\in \tau$, we define $f(\tau) := [T]$. 
	One easily verifies that $f(\tau)$ is well defined. 
	Since $\tau$ has a $G$-boundary type representative $\bd \Omega$ for some $\Omega\in \C^G(M)$, we see $f(\tau) = [\bd \Omega]\in \Z_n^G( M, \partial M; \mZ_2)$. 
	Thus we have a map $f: \Z_{n,rel}^G( M, \partial M; \mZ_2)\to \Z_n^G( M, \partial M; \mZ_2)$. 
	
	Secondly, for any $\sigma\in \Z_{n}^G( M, \partial M; \mZ_2)$, we can take an element $S\in \sigma\cap Z_{n,rel}( M, \partial M; \mZ_2)$ by Lemma \ref{Lem: equivalence in two formulation}, and define $h(\sigma) := [S]_{rel}$. 
	If $S,S'\in \sigma\cap Z_{n,rel}( M, \partial M; \mZ_2) $, then we have $S-S'\in \mR(\bd M;\mZ_2) \cap Z_{n,rel}( M, \partial M; \mZ_2)$, which implies $S\sim_{rel} S'$ and $h(\sigma)$ is well defined. 
	Note we can also take a $G$-boundary type representative $\bd \Omega'\in Z^G_{n,rel}(M, \partial M; \mZ_2)$ of $\sigma$ for some $\Omega'\in \C^G(M)$, and thus $h(\sigma) = [\bd \Omega']_{rel}$. 
	
	Using the $G$-boundary type representative, we immediately have $f\circ h =id$ and $h\circ f= id$. 
	Let ${\bf v} = \M$ or $\F$. 
	Then by (\ref{Eq: norm in two formulation}), we have ${\bf v}(f(\tau)) \leq {\bf v} (\tau) $ for all $\tau\in \Z_{n,rel}^G( M, \partial M; \mZ_2)$. 
	Meanwhile, for any $\sigma\in \Z_{n}^G( M, \partial M; \mZ_2)$, there is a sequence $\{S_i\}_{i=1}^\infty\subset \sigma$ so that ${\bf v}(S_i) \rightarrow {\bf v}(\sigma)$. 
	By Lemma \ref{Lem: equivalence in two formulation}, we can find $S_i'\in \sigma \cap Z_{n,rel}^G( M, \partial M; \mZ_2)$ for each $i\geq 1$ such that $\M(S_i - S_i') \leq \frac{1}{i}$. 
	Because $h(\sigma)=[S_i']_{rel}$ for every $i\geq 1$, we have ${\bf v}(h(\sigma)) \leq \lim_{i\to\infty} {\bf v}(S_i') = \lim_{i\to\infty} {\bf v}(S_i) = {\bf v}(\sigma)$. 
	Together, we see ${\bf v}(f(\tau)) = {\bf v}(\tau)$ and ${\bf v}(h(\sigma)) = {\bf v}(\sigma)$. 
\end{proof}

By the above proposition, we see the lower semi-continuous of mass \cite[Lemma 3.5]{wang2023min} and the compactness theorem \cite[Lemma 3.7]{wang2023min} are valid for the space $\Z_n^G(M,\partial M;\mZ_2)$ as well as $\Z_{n,rel}^G(M,\partial M;\mZ_2)$. 
Moreover, we also have the following isoperimetric lemma.
\begin{lemma}\label{Lem: isoperimetric lemma}{\rm (\cite[Lemma 3.9, 3.10]{wang2023min})} 
	There exist $\epsilon_M>0$ and $C_M>1$ depending only on the isometric embedding $M\hookrightarrow \R^L$, so that for any $\tau_1, \tau_2\in Z_n^G(M,\partial M;\mZ_2)/\sim$ with $ \F(\tau_1 - \tau_2) <\epsilon_M $,
	there is a unique $Q\in \mI_{n+1}^G(M;\mZ_2)$ 
	satisfying 
	\begin{equation*}
		T_2-T_1 -\partial Q \in \mR_n(\partial M; \mZ_2) \qquad {\rm and } \qquad \M(Q) \leq C_M \F(\tau_1 - \tau_2),
	\end{equation*}
	where $T_i$ is the canonical representative of $\tau_i$, $i=1,2$. 
	Moreover, if $ \M(\tau_1 - \tau_2) <\epsilon_M $, we also have
	\begin{equation*}
		T_2-T_1 = \partial Q + R \qquad {\rm and } \qquad \M(Q)+\M(R) \leq C_M \M(\tau_1 - \tau_2), 
	\end{equation*}
	for some unique $Q\in \mI_{n+1}^G(M;\mZ_2)$ and $R\in \mR^G_n(\bd M;\mZ_2)$.
\end{lemma}

One can see from the above lemma that $\Z_n^G(M,\bd M;\mZ_2)$ is the connected component of $0$ in $Z_n^G(M,\bd M;\mZ_2)/\sim$ with respect to the $\F$-topology.

\subsection{Area variations}

Now, let us introduce some notations for the variations of area functional. 
Denote by $\mfX(\R^L)$ the space of vector fields in $\R^L$. 
Then we define 
\begin{eqnarray*}
	\mfX(M) &:=& \left\{  X\in \mfX(\R^L) : X(p)\in T_pM \mbox{ for all } p\in M\right\},\\
	\mfX_{tan}(M) &:=& \left\{  X\in \mfX(M) : X(p)\in T_p\bd M \mbox{ for all } p\in \bd M\right\}. 
\end{eqnarray*}
Given $X\in \mfX_{tan}(M)$, the diffeomorphisms $\{f_t\}$ generated by $X$ satisfy $f_t(M) = M$. 
Thus, the first variation for a varifold $V\in \V_n(M)$ along $X\in \mfX_{tan}(M)$ is given by 
\[ \delta V(X) := \frac{d}{dt}\Big|_{t=0} \|(f_t)_\#V\|(M) = \int \Div_S X(x) dV(x,S).  \]
\begin{definition}\label{Def: FBMH}
	Given a relative open subset $U\subset M$, a varifold $V\in \V_n(M)$ is said to be {\em stationary in $U$ with free boundary} if $\delta V(X)=0$ for all $X\in \mfX_{tan}(M)$ compactly supported in $U$. 
\end{definition}

Next, consider a compact smooth embedded hypersurface $\Sigma\subset M$ with (possibly empty) boundary $\bd \Sigma$. 
If $\bd \Sigma\subset \bd M$, then we say $\Sigma$ is {\em almost properly embedded} in $M$ and denote by $(\Sigma,\bd\Sigma)\subset (M,\bd M)$. 
Note $\Sigma$ may touch $\bd M$ tangentially from inside. 
Thus, we denote by 
\[ \Scal(\Sigma) :=  (\Sigma\setminus\bd\Sigma) \cap \bd M \]
the {\em touching set} of $\Sigma$ in $M$. 
In particular, we say $\Sigma$ is {\em properly embedded} if $\Scal(\Sigma)=\emptyset$. 

For an almost properly embedded hypersurface $(\Sigma,\bd\Sigma)\subset (M,\bd M)$, define 
\[ \mfX(M,\Sigma) := \{X\in\mfX(M) : X(q)\in T_q\bd M \mbox{ for $q\in \bd M$ near $\bd \Sigma$} \}. \]
Then the first variation of $\Sigma$ along $X$ is given by 
\[ \delta\Sigma (X) =  \frac{d}{dt}\Big|_{t=0} \Area(f_t(\Sigma)) = -\int_\Sigma \langle H,X\rangle + \int_{\bd\Sigma} \langle\eta, X\rangle,  \]
where $H$ and $\eta$ are the mean curvature vector field and the outward unit co-normal of $\Sigma$. 
\begin{definition}
	Given a relative open subset $U\subset M$, an almost properly embedded hypersurface $(\Sigma,\bd\Sigma)\subset (M,\bd M)$ is said to be {\em stationary} in $U$ if $\delta\Sigma(X)=0$ for all $X\in \mfX(M,\Sigma)$ compactly supported in $U$. 
\end{definition}

Suppose $(\Sigma,\bd\Sigma)\subset (M,\bd M)$ is stationary in $M$. 
It then follows from the first variation formula that $H=0$ and $\Sigma$ meets $\bd M$ orthogonally along $\bd \Sigma$, which is known as a {\em free boundary minimal hypersurface} (abbreviated as FBMH). 
Additionally, let $\mfX^\perp(\Sigma)$ be the space of normal vector fields on $\Sigma$, and
\[\mfX^\perp(\Sigma\setminus\Scal(\Sigma))\]
be the space of normal vector fields on $\Sigma$ compactly supported in $\Sigma\setminus\Scal(\Sigma)$. 
Then for any $X\in \mfX^\perp(\Sigma\setminus\Scal(\Sigma))$, we can extend $X$ to a vector field in $\mfX(M,\Sigma)$ and compute the second variation of $\Sigma$ for the area functional:
\begin{equation}\label{Eq: second variation formula}
	\delta^2\Sigma (X) = Q(X,X) := \int_\Sigma \left( |\nabla^\perp X|^2 -\Ric_M(X,X) - |A_\Sigma|^2|X|^2 \right) - \int_{\bd\Sigma}h_{\bd M}(X,X) ,
\end{equation}
where $\Ric_M$ is the Ricci curvature of $M$, $A_\Sigma$ and $h_{\bd M}$ are the second fundamental forms of $\Sigma$ and $\bd M$ respectively. 
Then the {\em Morse index} ${\rm Index}(\Sigma)$ of $\Sigma$ is defined by the maximal dimension of a linear subspace in $\mfX^\perp(\Sigma\setminus\Scal(\Sigma))$ so that $Q(\cdot,\cdot)$ is negative definite on this subspace. 
Given a relative open subset $U\subset M$, 
if $Q(X,X)\geq 0$ for all $X\in \mfX^\perp(\Sigma)$ (resp. $X\in \mfX^\perp(\Sigma\setminus\Scal(\Sigma))$) compactly supported in $U$, then we say $\Sigma$ is {\em globally stable in $U$} (resp. {\em stable in $U$ away from the touching set $\Scal(\Sigma)$}). 

	


Meanwhile, we say an almost properly embedded hypersurface $(\Sigma,\bd\Sigma)\subset (M,\bd M)$ is $G$-invariant, if $ g\cdot \Sigma = \Sigma$ for all $g\in G$. 
Similarly, define 
\begin{itemize}
	\item $\mfX^G(M):=\{X\in \mfX(M) : dg(X)=X, \forall g\in G\}$;
	\item $\mfX^G(M,\Sigma) := \mfX^G(M)\cap \mfX(M,\Sigma)$;
	\item $\mfX^{\perp,G}(\Sigma) := \{X\in\mfX^\perp(\Sigma) : dg(X)=X, \forall g\in G \}$;
	\item $\mfX^{\perp,G}(\Sigma\setminus\Scal(\Sigma)) := \mfX^{\perp,G}(\Sigma)  \cap \mfX^\perp(\Sigma\setminus\Scal(\Sigma))$.
\end{itemize}
Then we can also define the equivariant Morse index for $G$-invariant FBMHs. 
\begin{definition}\label{Def: G-index}
	For an almost properly embedded $G$-invariant FBMH $(\Sigma,\bd\Sigma)\subset (M,\bd M)$, the {\em equivariant Morse index} (or $G$-index) ${\rm Index}_G(\Sigma)$ of $\Sigma$ is defined as the maximal dimension of a linear subspace in $\mfX^{\perp,G}(\Sigma\setminus\Scal(\Sigma))$ so that $Q(\cdot,\cdot)$ is negative definite on it. 
	
	Given a relative open $G$-subset $U\subset M$, 
	if $Q(X,X)\geq 0$ for all $X\in \mfX^{\perp,G}(\Sigma)$ (resp. $X\in \mfX^{\perp,G}(\Sigma\setminus\Scal(\Sigma))$) compactly supported in $U$, then we say $\Sigma$ is {\em globally $G$-stable in $U$} (resp. {\em $G$-stable in $U$ away from the touching set $\Scal(\Sigma)$}). 
\end{definition}

Note the global stability (resp. global $G$-stability) in the above definition is slightly stronger than $\delta^2\Sigma(X)\geq 0$ for all $X\in \mfX(M,\Sigma)$ (resp. $X\in \mfX^G(M,\Sigma)$) 
since $X\in \mfX^\perp(\Sigma)$ (resp. $X\in \mfX^{\perp,G}(\Sigma)$) may not admit an extension in $\mfX(M,\Sigma)$ (resp. $\mfX^G(M,\Sigma)$) unless $\Scal(\Sigma)=\emptyset$. 
Hence, the definitions of stability and $G$-stability in \cite{wang2023min} should be revised by global stability and global $G$-stability so that the arguments in \cite[Lemma 2.12]{wang2023min} would carry over, and the subsequent regularity theory is not affected. 
Indeed, the stable FBMHs concerned in the replacements constructions (\cite[\S 5]{li2021min}\cite[\S 5]{wang2023min}) are the limits of {\em properly embedded} stable local area minimizers, which are all globally stable. 

One can also refer to \cite{guang2021compactness} for the curvature estimates and the compactness theorem for almost properly embedded FBMHs that are stable away from the touching sets.

\section{Equivariant min-max theory and equivariant volume spectrum}\label{Sec: min-max setting}

For any $m\in\N$, let $I^m=[0,1]^m$ be the $m$-dimensional cube. 
Given $j\in\N$, define $I(1,j)$ to be the cube complex on $I^1$ with $0$-cells $\{[\frac{i}{3^{j}} ]\}_{i=0}^{3^j}$ and $1$-cells $\{[\frac{i}{3^{j}}, \frac{i+1}{3^{j}} ]\}_{i=0}^{3^j-1} $. 
Then the cube complex on $I^m$ is given by $I(m, j) := I(1,j)\otimes\cdots\otimes I(1,j)$ ($m$-times). 
We say $\alpha = \alpha_1\otimes	\cdots \otimes \alpha_m$ is a $q$-cell in $ I(m, j)$ if every $\alpha_i $ is a cell in $I(1,j)$ and $\sum_{i=1}^m\dim(\alpha_i ) = m$. 
Then for a cubical subcomplex $X$ of $I^m$, the cube complex $X(j)$ is the union of cells of $I(m, j)$ contained in $X$. 
We also use the notations $X(j)_p$ and $\alpha_p$ to denote the set of $p$-cells in $X(j)$ and $\alpha$ respectively. 

Given $x,y\in I(m,j)_0$, the {\em distance} between $x, y$ is defined by ${\bf d}(x,y) := 3^j\cdot\sum_{i=1}^m|x_i-y_i|$. 
We say $x,y$ are {\em adjacent} if ${\bf d}(x,y)=1$.
For any discrete map $\phi : I(m,j)_0 \to \Z_n^G(M, \partial M; \mZ_2)$, define 
the $\M$-fineness of $\phi$ by 
$\mf_\M := \sup \{\M(\phi(x)-\phi(y)) : x,y\in I(m,j)_0\mbox{ are adjacent}\}$. 

\subsection{Almgren's isomorphism for relative $G$-cycles space}

Given any $\F$-continuous closed curve $\Phi : [0,1]\to \Z_n^G(M, \partial M; \mZ_2)$ with $\Phi(0) = \Phi(1)$, the uniform continuity of $\Phi$ implies the existence of a number $K\in\N$ so that 
\begin{equation}\label{Eq: subdivided}
	\F \left( \Phi \Big( \frac{i+1}{3^{k}}  \Big) - \Phi \Big(\frac{i}{3^{k}} \Big) \right) < \epsilon_0 :=\frac{\epsilon_M}{3C_M} , \qquad \forall k\geq K, ~0\leq i \leq 3^k-1,
\end{equation}
where $\epsilon_M$ and $C_M$ are the constants in Lemma \ref{Lem: isoperimetric lemma}. 
Thus, for each $k\geq K$ and $0\leq i \leq 3^k-1$, we have a {\em unique} $A_{i}^k\in\mI_{n+1}^G(M;\mZ_2)$ so that $[\partial A_{i}^k] = \Phi ( \frac{i+1}{3^{k}} ) - \Phi (\frac{i}{3^{k}} )$ and $\M(A_i^k)\leq C_M\epsilon_0$. 
Since $\partial (\sum_{i=0}^{3^k-1} A_i^k) \in \mI_n^G(\partial M;\mZ_2)$, we can use constancy theorem and define 
\begin{equation}\label{Eq: almgren isomorphism}
	F_{M/\partial M}(\Phi) = \left[\sum_{i=0}^{3^k-1} A_i^k \right]\in H_{n+1}(M,\partial M; \mZ_2).
\end{equation}
(See \cite[\S 4.4]{federer2014geometric}.)
We claim this element $F_{M/\partial M}(\Phi)$ is well defined. 
Indeed, for every $k\geq N$ and $0\leq i\leq 3^{k}-1$, we have $[\partial (A^{k+1}_{3i} + A^{k+1}_{3i+1} + A^{k+1}_{3i+2})] = [\partial A^k_i ]$ and $\M(A^{k+1}_{3i} + A^{k+1}_{3i+1} + A^{k+1}_{3i+2}) \leq 3C_M\epsilon_0 < \epsilon_M$. 
The uniqueness of the isoperimetric choice (Lemma \ref{Lem: isoperimetric lemma}) implies $A^k_i = \sum_{j=0}^2A^{k+1}_{3i+j}$, and thus $F_{M/\partial M}(\Phi)$ is well-defined. 

In addition, $F_{M/\partial M}$ also induces a well-defined homomorphism 
\[F_{M/\partial M} : \pi_1(\Z_n^G(M, \partial M; \mZ_2); \{0\} ) \to H_{n+1}(M,\partial M;\mZ_2) .\]
Indeed, if $\Phi' : [0,1]\to \Z_n^G(M, \partial M; \mZ_2)$, $\Phi'(0) = \Phi'(1)$, is another closed curve that is homotopic to $\Phi$ in $\Z_n^G(M, \partial M; \mZ_2)$ relative to the boundary value. 
Then one can apply the above discretization and cobordism method to the relative homotopy map $H: I^2\to \Z_n^G(M, \partial M; \mZ_2)$ and verify that $F_{M/\partial M}(\Phi) = F_{M/\partial M}(\Phi') $. 
Moreover, as \cite[Remark 2]{wang2022min}, we also have the following isomorphisms. 

\begin{theorem}\label{Thm: Almgren isomorphism}
	For an $(n+1)$-dimensional Riemannian manifold $(M,g_{_M})$ with a compact Lie group $G$ acting by isometries, we have 
	\[\pi_1(\Z_n^G(M, \partial M; \mZ_2); \{0\} )\cong \pi_1(\Z_n^G(M, \partial M;\M; \mZ_2) ; \{0\} ) \cong H_{n+1}(M,\partial M;\mZ_2) \cong\mZ_2.\]
	Additionally, $\pi_m(\Z_n^G(M, \partial M; \mZ_2);\{0\}) = \pi_m(\Z_n^G(M, \partial M; \M; \mZ_2); \{0\}) = 0 $ for $m\geq 2$.  
\end{theorem}
\begin{proof}
	Take any $G$-equivariant Morse function $f:M\to [0,1]$ (in the sense of \cite{wasserman1969equivariant}). 
	Then we have a $\F$-continuous curve $\Phi: [0,1]\to \Z_n^G(M, \partial M; \mZ_2) $, $\Phi(t) = [\partial \{x\in M: f(x)<t\} ]$ (c.f. \cite[Claim 5.6]{marques2017existence}). 
	Thus, $F_{M/\partial M}(\Phi) = [M]\in  H_{n+1}(M,\partial M;\mZ_2)$, and $F_{M/\partial M}$ is surjective. 
	Additionally, by a discretization/interpolation procedure using \cite[Theorem 4.11, 4.13]{wang2023min}, we also have $F_{M/\partial M}$ is surjective with respect to the $\M$-topology. 

	Meanwhile, we can use the $G$-invariant isoperimetric lemma \ref{Lem: isoperimetric lemma} in place of \cite[Proposition 1.22,1.23]{almgren1962homotopy} to obtain the existence of chain maps in $\mI^G_*(M)$ parallel to \cite[Theorem 2.5]{almgren1962homotopy}. 
	Then, we can replace the cutting/deforming maps in \cite[\S 5]{almgren1962homotopy} with the equivariant ones in the proof of \cite[Theorem 4.13]{wang2023min}, which further gives an $\M$-continuous $G$-equivariant interpolation theorem in $\Z_n^G(M,\bd M;\mZ_2)$ parallel to \cite[Theorem 6.6]{almgren1962homotopy}. 
	Finally, combining the above equivariant ingredients with the homotopy constructions in \cite[\S 7.2, Theorem 7.5]{almgren1962homotopy}, we have $F_{M,\bd M}$ is an isomorphism and $\pi_m(\Z_n^G(M, \partial M; \mZ_2)) = 0 $ for $m\geq 2$. 
	Since the interpolation maps constructed in \cite[Theorem 4.13]{wang2023min} are $\M$-continuous, the isomorphisms are also valid with respect to the $\M$-topology. 
\end{proof}

It should be noted that the constructions in \cite{almgren1962homotopy} are formulated with integral currents. 
Nevertheless, by Proposition \ref{Prop: equivalence between two formulation} (see also \cite[Proposition 3.2]{guang2021min}), it is equivalent to use $\Z_n^G(M,\bd M;\mZ_2)$ which is formulated by integer rectifiable $\mZ_2$-currents. 
In addition, the above arguments also give the following homotopy theorem parallel to \cite[Theorem 8.2]{almgren1962homotopy}. 


\begin{theorem}\label{Thm: Almgren homotopy}
	For any neighborhood $\mathcal{N}$ of $0$ in $\Z_n^G(M,\partial M;\mZ_2)$ with the $\F$-topology, there exist positive numbers $\epsilon_{-1}>\epsilon_0>\epsilon_1>\cdots$ and $1<C_0<C_1<\cdots$, so that $\mB^\F_{\epsilon_{-1}}(0) := \{\tau\in \Z_n^G(M,\partial M;\mZ_2): \F(\tau)\leq \epsilon_{-1}\}\subset \mathcal{N}$, and for any $m\in \N$, if $\Phi : I^m\to \Z_n^G(M,\partial M;\mZ_2)$ is an $\F$-continuous map with $\Phi\llcorner\partial I^m \equiv 0$ and $\sup_{x\in I^m}\F(\Phi(x)) \leq \epsilon_m$, then there is a relative homotopy $H: I^{m+1} \to \Z_n^G(M,\partial M;\mZ_2)$ satisfying
	\begin{itemize}
		\item $H$ is continuous in the $\F$-topology;
		\item $H(0,x) \equiv 0$ and $H(1,x)= \Phi(x)$ for all $x\in I^m$;
		\item $H(s, x)\equiv 0$ for all $x\in\partial I^m$ and $s\in I$;
		\item $\sup \{ \F(H(s,t)) : s\in I, x\in I^m\} \leq C_m\epsilon_m \leq \epsilon_{m-1} $.
	\end{itemize} 
\end{theorem}
\begin{proof}
	Let $\epsilon_M$ be given in Lemma \ref{Lem: isoperimetric lemma}. Then we take $0<\epsilon_{-1} < \epsilon_M/2 $ so that $\mB^\F_{\epsilon_{-1}}(0) \subset \mathcal{N}$. 
	For any $m\in\N$, suppose $\Phi : I^m\to \Z_n^G(M,\partial M;\mZ_2)$ is an $\F$-continuous map with $\Phi\llcorner\partial I^m \equiv 0$ and $\sup_{x\in I^m}\F(\Phi(x)) \leq \epsilon_{-1}$. 
	In particular, we have $\F(\Phi(x) - \Phi(y)) \leq \epsilon_M$ for all $x,y\in I^m$. 
	As in the proof of Theorem \ref{Thm: Almgren homotopy}, we can combine the equivariant interpolation constructions in \cite[Theorem 4.13]{wang2023min} with the homotopy constructions in \cite[\S 7.2, Theorem 7.5]{almgren1962homotopy} to get a relative homotopy $H$ between $\Phi$ and the constant map $x\mapsto 0$ in $\Z_n^G(M,\bd M;\mZ_2)$ so that $\sup_{(s,x)\in I^{m+1}}\F(H(s,x)) \leq C_m\sup_{x\in I^{m}}\F(\Phi(x))$, where $C_m>1$ is a positive number depending only on $M,G$ and $m$. 
	Finally, the theorem follows from taking $0< \epsilon_m \leq \epsilon_{m-1}/C_m$ and repeating the above procedure for every $m\in\N$ inductively. 
\end{proof}

As an application, we have the following result parallel to \cite[Proposition 3.3]{marques2017existence}.
\begin{proposition}\label{Prop: homotopy between flat close maps}
	Let $X$ be a cubical subcomplex of some $I(m,j)$. Then there exists $\delta=\delta(M,G,m)>0$, so that if $\Phi_1,\Phi_2: X \to \Z_n^G(M,\partial M;\mZ_2)$ are $\F$-continuous maps with 
	\[\sup \{ \F(\Phi_1(x) - \Phi_2(x)) : x\in X \} < \delta,\]
	then $\Phi_1$ is homotopic to $\Phi_2$ in $\Z_n^G(M,\partial M;\mZ_2)$ with respect to the $\F$-topology. 
\end{proposition}
\begin{proof}
	After replacing $\Z_n(M;\mZ_2)$ and \cite[Theorem 8.2]{almgren1962homotopy} by $\Z^G_n(M,\bd M;\mZ_2)$ and Theorem \ref{Thm: Almgren homotopy}, the proof can be taken almost verbatim from \cite[Proposition 3.3]{marques2017existence}. 
\end{proof}

\subsection{Equivariant volume spectrum}

Now we define the $(G,p)$-sweepouts and $(G,p)$-widths similarly to \cite{liokumovich2018weyl}. 
By the isomorphisms in Theorem \ref{Thm: Almgren isomorphism}, we have 
\[ H^1( \Z_n^G(M,\partial M; \mZ_2); \mZ_2) = \mZ_2 = \{0, \bar{\lambda}\} .\]
Hence, for any $\kappa\in H_1( \Z_n^G(M,\partial M; \mZ_2); \mZ_2)$ with an representative $\Phi: S^1 \to \Z_n^G(M,\partial M; \mZ_2)$, we have $\bar{\lambda}(\kappa) = 1$ if and only if $F_{M/\partial M}([\Phi]) = [M]$. 
\begin{definition}[$(G,p)$-sweepout]\label{Def: p-sweepouts}
	Given an integer $p\geq 1$, we say an $\F$-continuous map $\Phi:X\rightarrow \mathcal{Z}_n^G(M,\partial M;\mathbb{Z}_2)$ is a {\em  $(G,p)$-sweepout} of $M$ if 
	\[\Phi^*(\bar \lambda^p) \neq 0 \in H^p(X;\mathbb{Z}_2),\]
	where $\bar{\lambda}^p :=\bar{\lambda}\smile\dots\smile\bar{\lambda}$ is the cup product of $\bar{\lambda}$ with itself for $p$-times. 
\end{definition}

By the isomorphism in Theorem \ref{Thm: Almgren isomorphism}, if an $\F$-continuous map $\Phi'$ is homotopic in $\mathcal{Z}_n^G(M,\partial M;\mathbb{Z}_2)$ to a $(G,p)$-sweepout of $M$, then $\Phi'$ is also a $(G,p)$-sweepout of $M$.

For any $\F$-continuous map $\Phi: X\rightarrow \mathcal{Z}_n^G(M,\partial M;\mathbb{Z}_2)$, define
\[{\bf m}^G(\Phi,r) :=\sup\left\{\|T_x\|( \tB_r(G\cdot p)) : x\in X,~p\in M  \right\},\]
where $T_x$ is the canonical representative of $\Phi(x)$, and $\tB_r(G\cdot p)$ is the open geodesic tube in $\tM\supset M$ of radius $r$ around $G\cdot p$. 
We say $\Phi$ {\it has no concentration of mass on orbits} if 
\[\lim_{r\rightarrow 0}{\bf m}^G(\Phi,r)=0.\]
The following lemma indicates this is a mild technical condition:
\begin{lemma}\label{Lem: no concentration of mass on orbits}{\rm (\cite[Lemma 4.11]{wang2023min})}
	If $\Phi: X\to \Z_{n}^G(M, \partial M;\mZ_2)$ is $\mF$-continuous or $\M$-continuous, then $\sup_{x\in X}\M(\Phi(x)) < \infty$ and $\Phi$ has no concentration of mass on orbits. 
\end{lemma}

\begin{definition}\label{Def: p-width}
	For any integer $p\geq 1$, let $\mathcal{P}_p^G(M)$ be the set of all $(G,p)$-sweepouts with no concentration of mass on orbits. 
	Then the {\em $(G,p)$-width} of $(M, g_{_M})$ is define by 
	\begin{equation}\label{Eq: p-width}
		\omega^G_p(M, g_{_M}) := \inf_{\Phi\in \mathcal{P}_p^G(M)} \sup_{x\in {\rm dmn}(\Phi)} \M(\Phi(x)),
	\end{equation}
	where ${\rm dmn}(\Phi)$ is the domain of $\Phi$, and the mass $\M$ is given under the metric $g_{_M}$. 
\end{definition}

\begin{remark}\label{Rem: equivalence between p-widths definition}
	By Proposition \ref{Prop: equivalence between two formulation}, one can also use $\Z_{n, rel}^G(M,\partial M;\mZ_2) $ in the definitions of $\mathcal{P}_p^G(M)$ and $\omega_p^G(M, g_{_M})$. 
	In particular, if we take $G=\{id\}$, then the $(G,p)$-width coincides with the $p$-width defined in \cite[(7)]{liokumovich2018weyl}. 
\end{remark}

In \cite{wang2022min}, the author has proved that if $\partial M =\emptyset$, then $p^{-\frac{1}{l+1}}\omega_p^G(M, g_{_M}) = O(1)$, where $l+1 = {\rm Cohom}(G)\geq 2$. 
For the case that $\partial M\neq\emptyset$, this estimate 
is also valid by a similar argument in \cite[\S 8, 9]{wang2022min} using lemmas in \cite[\S 3.1]{wang2023min}.

In the following proposition, we show that it is sufficient to consider only the $\M$-continuous maps in $\mathcal{P}_p^G(M)$ in the computation of $\omega_p^G(M, g_{_M})$. 

\begin{proposition}\label{Prop: p-width using M-continuous map}
	For any $\Phi\in \mathcal{P}_p^G(M)$ and $\epsilon >0$, there exists $\Phi'\in \mathcal{P}_p^G(M)$ continuous in the $\M$-topology such that ${\rm dmn}(\Phi') = {\rm dmn}(\Phi)$ and 	
	$$ \sup_{x\in {\rm dmn}(\Phi)}\M(\Phi'(x)) \leq \sup_{x\in {\rm dmn}(\Phi)} \M(\Phi(x)) + \epsilon . $$
\end{proposition}
\begin{proof}
	Given $\Phi \in \mathcal{P}_p^G(M)$ with $X={\rm dmn}(\Phi)\subset I^m$, we can apply the discretization theorem \cite[Theorem 4.11]{wang2023min} to $\Phi$, and obtain a sequence of discrete maps $\{\phi_i\}_{i\in\N}$, 
	$\phi_i: X(l_i)_0 \to \Z_{n}^G(M, \partial M;\mZ_2)$
	with $\{l_i\}_{i\in\N}\subset \N$ and $l_i<l_{i+1}$ such that 
	\begin{itemize}
		\item $\mf_{\M}(\phi_i) < \delta_i$;
		\item $\sup \{ \F(\phi_i(x) - \Phi(x)) : x\in X(l_i)_0  \} \leq \delta_i $;
		\item $ \sup\{\M (\phi_i(x)) : x\in X(l_i)_0\} \leq \sup \{ \M (\Phi(x)) : x\in X\}+\delta_i$;
	\end{itemize}
	where $\delta_i \to  0$ is a sequence of positive numbers. 
	Then for sufficiently large $i$, we apply the interpolation theorem \cite[Theorem 4.13]{wang2023min} to $\phi_i$, and obtain an $\M$-continuous map $\Phi_i: X\to \Z_{n}^G(M, \partial M;\mZ_2)$ so that $\Phi_i(x) = \phi_i(x)$ for all $x\in X(l_i)_0$ and 
	\[\sup\{ \M(\Phi_i(x) - \Phi_i(y)) : x,y \in\alpha_0 {\rm ~for~some~}\alpha\in  X(l_i)_m\} \leq C\mf_{\M}(\phi_i ),\]
	where $C = C(M,G,m)>1$. 
	Let $\delta >0$ be the positive number in Proposition \ref{Prop: homotopy between flat close maps}, and $\epsilon>0$ be arbitrary. 
	Then we have $ 2C\delta_i < \min\{\epsilon, \delta \}$ for $i$ large enough. 
	Together, we see
	\begin{itemize}
		\item $\sup_{x\in X}\M(\Phi_i(x)) \leq \sup_{x\in X} \M (\Phi(x)) + \delta_i+C\mf_{\M}(\phi_i ) \leq \sup_{x\in X} \M (\Phi(x)) + \epsilon$;
		\item $\sup_{x\in X} \F(\Phi_i(x) - \Phi(x)) \leq \delta_i + C\mf_{\M}(\phi_i ) < \delta$.
	\end{itemize}
	The proof is finished by Proposition \ref{Prop: homotopy between flat close maps} and taking $\Phi' := \Phi_i$ for $i$ large enough. 
\end{proof}

Recall that a Riemannian metric $g_{_M}'$ on $M$ is said to be {\em $G$-invariant} if $g^*g_{_M}' = g_{_M}' $ for all $g\in G$ (or equivalently $G\subset {\rm Diff}(M)$ acts by isometrics on $(M, g_{_M}')$). 
Then we state the following lemma parallel to \cite[Lemma 2.1]{irie2018density}. 
\begin{lemma}\label{Lem: Gp-width depends continuously on metric}
	For any $p\in\mZ^+$, the $(G,p)$-width $\omega^G_p(M, g_{_M})$ depends continuously on the $G$-invariant metric $g_{_M}$ (in the $C^0$ topology). 
\end{lemma}
\begin{proof}
	The proof of \cite[Lemma 2.1]{irie2018density} would carry over with $G$-invariant metrics $g_{_M}^i$ and $(G,p)$-width $\omega^G_p(M, g_{_M}^i)$ in place of $g_i$ and $\omega_k(M,g_i)$. 
\end{proof}

\subsection{Equivariant min-max hypersurfaces}

We now collect some results in equivariant min-max theory proved by the author (\cite{wang2022min}\cite{wang2023min}\cite{wang2023equivariant}) in Almgren-Pitts' setting. 
In this subsection, the Riemannian manifold $(M^{n+1}, g_{_M})$ and the Lie group $G$ actions are assumed to satisfy either (\ref{Eq: assumption closed}) or (\ref{Eq: assumption with boundary}), 
which are the dimension assumptions used in the closed case \cite{wang2022min}\cite{wang2023equivariant} and free boundary case \cite{wang2023min} respectively. 

Firstly, we introduce some notations for min-max hypersurfaces. 
For simplicity, we also unambiguously regard closed minimal hypersurfaces as FBMHs. 

\begin{definition}[Good $G$-replacement property]\label{Def: good replacements property}
	Let $U\subset M$ be a relative open $G$-set and $V_0\in\V^G_n(M)$ be a $G$-varifold that is stationary in $U$ with free boundary. 
	We say $V_0$ has {\em good $G$-replacement property} in $U$, if for any finite sequence of compact $G$-sets $\{K_i\subset U\}_{i=1}^q$, there are $G$-varifolds $\{V_i\}_{i=1}^q\subset\V_n^G(M)$ stationary in $U$ with free boundary so that 
	\[V_{i-1} \llcorner (M\setminus K_i) = V_i \llcorner (M\setminus K_i),\quad \|V_0\|(M) = \|V_i\|(M), \quad  V_i\llcorner \interior_M (K_i)= |\Sigma_i|,  \]
	where $1\leq i\leq q$ and $(\Sigma_i,\bd\Sigma_i)\subset (\interior_M(K), \interior_M(K)\cap\bd M)$ is an almost properly embedded $G$-invariant (globally) stable FBMH with integer multiplicity. 
\end{definition}

For any $p\in M$, fix $r_0=r_0(G\cdot p)>0$ sufficiently small so that
\begin{itemize}
	\item $\tB^G_{r_0}(p)\subset \mB_{r_0}(G\cdot p)\cap M \subset  \tB^G_{\inj(G\cdot p)}(p)\cap(M\setminus \bd M)$ for $G\cdot p \subset M\setminus \bd M$;
	\item $\mB_{r_0}(G\cdot p)\cap M \subset \tB^G_{\inj(G\cdot p)}(p)\cap \tBcal^G_{r^G_{Fermi}(p)}(p)$ for $G\cdot p\subset \bd M$,
\end{itemize}
where $\tBcal^G_{r^G_{Fermi}(p)}(p)$ is the $G$-invariant Fermi half tube given by \cite[Lemma B.6]{wang2023min}. 
\begin{definition}[Min-max $(c,G)$-hypersurfaces with free boundary]\label{Def: min-max $(c,G)$-hypersurfaces}
	Let $c\in\mZ_+$ and $\Sigma\in \V^G_n(M)$ be stationary in $M$ with free boundary. 
	Then $\Sigma$ is said to be a {\em free boundary min-max $(c,G)$-hypersurface (with multiplicity)} if for any $c$ concentric $G$-annuli 
	\[\mathcal{A}=\{\mAn^L_{s_i,t_i} (G\cdot p)\cap M\}_{i=1}^c \quad {\rm with} \quad 0<2 s_i < 2t_i < s_{i+1} < t_{i+1}<r_0(G\cdot p), \]
	$\Sigma$ has good $G$-replacement property in at least one of the $G$-annulus in $\mathcal{A}$. 
	
	Generally, we say $\Sigma$ is a {\em free boundary min-max $G$-hypersurface (with multiplicity)} if it is a free boundary min-max $(c, G)$-hypersurface for some $c \in  \mZ_+$.
\end{definition}

\begin{remark}
	By the min-max regularity results in \cite{wang2022min}\cite{wang2023min}\cite{wang2023equivariant}, it is reasonable to call the $G$-varifold $\Sigma\in \V^G_n(M)$ in Definition \ref{Def: min-max $(c,G)$-hypersurfaces} a $G$-hypersurface since it is induced by an almost properly embedded $G$-invariant FBMH with multiplicity provided (\ref{Eq: assumption closed}) or (\ref{Eq: assumption with boundary}). 
	Hence, we sometimes abuse the notations and use $\Sigma, \bd\Sigma$ to represent $\spt(\|\Sigma\|), \bd \spt(\|\Sigma\|)$. 
\end{remark}
Note the above definition is generalized from \cite[Definition 5.1]{wang2023equivariant} where the closed manifolds and closed minimal $G$-hypersurfaces are considered. 
In particular, if $\bd M=\emptyset$, all the FBMHs turn out to be closed. 
Using the technique in \cite[Theorem 5.3]{wang2023equivariant}, we have the following compactness result for free boundary min-max $(c,G)$-hypersurfaces. 
\begin{theorem}[Compactness Theorem for free boundary min-max $G$-hypersurfaces]\label{Thm: compactness theorem}
	Let $(M^{n+1},g_{_M})$ be a compact Riemannian manifold with (possibly empty) boundary $\bd M$ and $G$ be a compact Lie group acting by isometries on $M$ so that either (\ref{Eq: assumption closed}) or (\ref{Eq: assumption with boundary}) is satisfied. 
	Suppose $c\in\mZ_+$, $\{\Sigma_i \}_{i\in\N}$ is a sequence of free boundary min-max $(c,G)$-hypersurfaces (with multiplicities) so that $\sup_{i\in\N} \|\Sigma_i\|(M) \leq C <+\infty$. 
	Then $\Sigma_i $ converges (up to a subsequence) in the varifold sense to a free boundary min-max $(c,G)$-hypersurface $\Sigma_\infty$ (with multiplicity). 
	Additionally, there exists a finite union of orbits $\mathcal{Y}=\cup_{k=1}^K G\cdot p_k$ such that the convergence (up to a subsequence) is locally smooth and graphical on $\Sigma_\infty\setminus \mathcal{Y} $ with integer multiplicity. 
	
	Moreover, if $\Sigma_i$ is a free boundary min-max $(c,G)$-hypersurface under the metric $g_{_M}^i$ so that $g_{_M}^i\to g_{_M}$ in the smooth topology, then the above compactness result remains valid. 
\end{theorem}
\begin{proof}
	If $M$ and $G$ are in the case of (\ref{Eq: assumption closed}), then the above theorem follows directly from \cite[Theorem 5.3, Remark 5.5]{wang2023equivariant}. 
	For the case of (\ref{Eq: assumption with boundary}), one can first use the free boundary min-max regularity theory in \cite[\S 5]{wang2023min} to show the conclusions in \cite[Theorem 4.18, Proposition 4.19]{wang2023equivariant} are also valid in the free boundary case. 
	After replacing \cite[Corollary 1]{schoen1981regularity} by the compactness theorem for stable FBMHs \cite{guang2021compactness}, the theorem then follows from the process of taking subsequence as in \cite[Theorem 5.3]{wang2023equivariant}.  
	Finally, we also have the last statement as in \cite[Remark 5.5]{wang2023equivariant} since the compactness theorem in \cite{guang2021compactness} is also valid for varying metrics. 
\end{proof}

Combining the above compactness result for min-max $G$-hypersurfaces with the free boundary min-max constructions in \cite{wang2023min}, we obtain the following proposition showing the $(G,p)$-width can be realized by the area of some free boundary min-max $G$-hypersurfaces with multiplicities. 
\begin{theorem}\label{Thm: Gp-width realization}
	Suppose $M$ and $G$ satisfy (\ref{Eq: assumption closed}) or (\ref{Eq: assumption with boundary}). 
	Then for any $p\in\mZ^+$, there exist $I\in\N$, $\{m_i\}_{i=1}^I\subset \N$, and a disjoint collection $\{\Sigma_i\}_{i=1}^I$ of almost properly embedded $G$-invariant FBMHs such that 
	\[\omega^G_p(M,g_{_M}) = \sum_{i=1}^I m_i\cdot {\rm Area}_{g_{_M}}(\Sigma_i) , 
	\]
	and $\sum_{i=1}^I m_i|\Sigma_i|$ is a free boundary min-max $((3^p)^{3^p},G)$-hypersurface. 
\end{theorem}
\begin{proof}
	For $M$ and $G$ satisfying (\ref{Eq: assumption closed}), the proposition follows directly from \cite[Corollary 1.7]{wang2023equivariant}. 
	For the case of (\ref{Eq: assumption with boundary}), the arguments are similar. 
	Firstly, by Proposition \ref{Prop: p-width using M-continuous map}, we can take a sequence $\Phi_i: X_i \to \Z_n^G(M,\bd M; \mF; \mZ_2)$ in $\mathcal{P}_p^G(M)$ so that $ \sup_{x\in X_i } \M(\Phi_i(x)) \to \omega_p^G(M,g_{_M})$. 
	Denote by $X_i^{(p)}$ the $p$-skeleton of $X_i$, and by $\iota :X^{(p)}\to X$ the inclusion map. 
	Then, noting $H^p(X_i, X^{(p)}_i;\mZ_2) = 0$, it follows from the exact cohomology sequence
	$ H^{p}(X_i, X_i^{(p)} ; \mZ_{2}) \stackrel{j^{*}}{\rightarrow} H^{p}\left(X_i ; \mZ_{2}\right) \stackrel{\iota^{*}}{\rightarrow} H^{p}(X_i^{(p)} ; \mZ_{2}) $  
	that $\Psi_i:= \Phi_i\circ \iota \in \mathcal{P}_p^G(M)$ is also a $(G,p)$-sweepout with no concentration of mass on orbits (Lemma \ref{Lem: no concentration of mass on orbits}) so that $\sup_{x\in X_i^{(p)} } \M(\Psi_i(x)) \to \omega_p^G(M,g_{_M})$. 
	Let $	\mathbf{\Pi}_i$ be the homotopy class of $\Psi_i$ containing all the maps $\Psi_i': X_i^{(p)} \to \Z_n^G(M,\bd M; \mF; \mZ_2)$	homotopic to $\Psi_i$ in the flat topology. 
	One easily verifies that $\mathbf{\Pi}_i\subset \mathcal{P}_p^G(M)$ and 
	\begin{equation}\label{Eq: realize width}
		 \omega_p^G(M,g_{_M}) \leq \mathbf{L}(\mathbf{\Pi}_i) := \inf_{\Psi_i'\in \mathbf{\Pi}_i}\sup_{x\in X_i^{(p)}} \M(\Psi_i'(x)) \leq \sup_{x\in X_i^{(p)} } \M(\Psi_i(x)) \to  \omega_p^G(M,g_{_M}). 
	\end{equation}
	Next, combining the free boundary min-max constructions in \cite{wang2023min} with the arguments in \cite[\S 4]{wang2023equivariant}, we can make the $G$-varifold $V_i$ associated with $\mathbf{\Pi}_i$ in the free boundary equivariant min-max theorem (\cite[Theorem 4.20]{wang2023min}) to be a min-max $((3^p)^{3^p},G)$-hypersurface with free boundary (Definition \ref{Def: min-max $(c,G)$-hypersurfaces}). 
	Finally, using the compactness theorem \ref{Thm: compactness theorem}, we see $V_i$ converges (up to a subsequence) in the varifold sense to a free boundary min-max $((3^p)^{3^p},G)$-hypersurface with multiplicity whose area realizing $\omega^G_p(M,g_{_M})$ by (\ref{Eq: realize width}). 
\end{proof}

\begin{remark}\label{Rem: index FBMHs}
	In \cite[Corollary 1.7]{wang2023equivariant}, the author also showed $\sum_{i=1}^I{\rm Index}_G(\Sigma_i)\leq p$ under the assumption (\ref{Eq: assumption closed}). 
	Meanwhile, using the equivariant free boundary bumpy metric theorem (Proposition \ref{Prop: G-bumpy is dense}(i)), one can generalize the arguments in \cite[\S 6]{wang2023equivariant}\cite{guang2021min} to the free boundary case (\ref{Eq: assumption with boundary}) and show the same $G$-index upper bounds for the $G$-invariant FBMHs in Theorem \ref{Thm: Gp-width realization}. 
	As this result is not required for our purpose, we leave the details to other places. 
\end{remark}

\begin{remark}\label{Rem: singular MHs}
	If $\bd M =\emptyset$ and ${\rm codim}(G\cdot x)\geq 3$ for all $x\in M$. 
	It follows from the regularity result \cite[Theorem 4.18]{wang2023equivariant} that the $G$-varifold $\Sigma$ in Definition \ref{Def: min-max $(c,G)$-hypersurfaces} is induced by a closed $G$-invariant minimal hypersurface with multiplicity that is smooth away from a ($G$-invariant) singular set of Hausdorff dimension no more that $n-7$. 
	Hence, using \cite{schoen1981regularity}, the above conclusions (Theorem \ref{Thm: compactness theorem} and \ref{Thm: Gp-width realization}) are also valid once we allow the closed $G$-invariant minimal hypersurfaces to possess such small singular sets. 
\end{remark}

\section{Weyl law for equivariant volume spectrum}\label{Sec: weyl law}

In this section, we study the Weyl law for the equivariant volume spectrum. 
The arguments are separated into two cases. 
Firstly, we consider $G$-manifolds with a single orbit type, i.e. $M=M^{prin}$, and use the co-area formula to reduce estimates into the smooth orbit space $M/G$. 
Next, for general $G$-manifolds, we can cut off a small tubular neighborhood of $M\setminus M^{prin}$ to obtain a compact domain $M'$ in $M^{prin}$ with piecewise smooth boundary. 
By the density of $M^{prin}$, we can use the $(G,p)$-width of $M'$ to approximate the $(G,p)$-width of $M$. 

Throughout this section, we always assume ${\rm Cohom}(G)=l+1\geq 2$.

\subsection{Weyl law in $G$-manifolds with a single orbit type}\label{Subsec: weyl law in single orbit type}
In this subsection, we consider the case that there exists only a single orbit type in $M$, i.e. $M=M^{prin}$. 
Therefore, there is a Riemannian metric $g_{_{M/G}}$ on $M/G$ so that $\pi: M\to M/G$ is a Riemannian submersion. 
For simplicity, denote by 
\begin{equation}\label{Eq: orbit volume}
	\vartheta(G\cdot p) = \vartheta([p]) := \mH^{n-l}(G\cdot p)
\end{equation}
the volume functional of orbits. 

To begin with, we show the following lemma, which indicates that $\pi$ maps a $G$-invariant rectifiable set in $M$ to a rectifiable set in $M/G$ with the same codimension. 

\begin{lemma}\label{Lem: reduce rectifiable G-set}
	For any $m$-rectifiable $G$-set $\widetilde{\Sigma}\subset\subset M^{prin}$, we have $\pi(\widetilde{\Sigma})\subset M^{prin}/G$ is an $(l-n+m)$-dimensional rectifiable set with $\mH^m(\widetilde{\Sigma}) = \int_{\pi(\widetilde{\Sigma})} \vartheta([q])~d\mH^{l-n+m}([q])$. 
\end{lemma}
\begin{proof}
	By \cite[Theorem 11.6]{simon1983lectures}, for $\mH^m$-a.e. $p\in\widetilde{\Sigma}$, the approximate tangent space $T_p\widetilde{\Sigma}$ exists. 
	Then we claim that 
	\begin{equation}\label{Eq: tangent space split}
		T_p\widetilde{\Sigma} =  T_pG\cdot p \times P_p,
	\end{equation}
	where $P_p\subset {\bf N}_pG\cdot p $ is an $(l-n+m)$-subspace. 
	Indeed, given any $w\in T_pG\cdot p$, there is a curve $g(t)\subset G$ so that $g(0)= e$ and $\frac{d}{dt}\big\vert_{t=0} g(t)\cdot p=w$. 
	Since $G$ acts orthogonally on $\R^L$, we write $g(t)$ as matrix $g(t) = I + tA(t)$ with $A(0)\cdot p = w$. 
	Let $\bleta_{p,r}:\R^L\to\R^L$ be the blowup map given by $\bleta_{p,r}(x)=(x-p)/r$. 
	Then, 
	\begin{eqnarray*}
			(\bleta_{p, \lambda}\circ g(\lambda))(x) = \frac{g(\lambda)\cdot x - g(\lambda)\cdot p}{\lambda}+\frac{g(\lambda)\cdot p -  p}{\lambda} = \big(\btau_{-A(\lambda)\cdot p}\circ g(\lambda)\circ \bleta_{p, \lambda} \big)(x),\nonumber
		\end{eqnarray*}
	where $\btau_{y}:\R^L\to\R^L$ is the translation $\btau_{y}(x) = x-y$. 
	Since $\btau_{-A(\lambda)\cdot p}$ and $g(\lambda)$ are isometries, we have the following result by the area formula and the definition of $T_p\widetilde{\Sigma}$: 
	\begin{eqnarray*}
		\int_{T_p\widetilde{\Sigma}} f ~d\mH^m  &=& \lim_{\lambda\to 0}\int_{\bleta_{p, \lambda}(\widetilde{\Sigma})} f ~d\mH^m =  \lim_{\lambda\to 0} \int_{\bleta_{p, \lambda}(g(\lambda)\cdot \widetilde{\Sigma})} f ~d\mH^m  
		\\&=&  \lim_{\lambda\to 0}\int_{\btau_{-A(\lambda)\cdot p}\circ g(\lambda)\circ \bleta_{p, \lambda} (\widetilde{\Sigma})} f ~d\mH^m
		\\&=&  \lim_{\lambda\to 0}\int_{\bleta_{p, \lambda}(\widetilde{\Sigma})} f\circ \btau_{-A(\lambda)\cdot p}\circ g(\lambda) ~d\mH^m
		\\&=& \int_{T_p\widetilde{\Sigma}} f\circ\btau_{-w} ~d\mH^m = \int_{\btau_{-w}(T_p\widetilde{\Sigma})} f ~d\mH^m, 
	\end{eqnarray*}
	for all $f\in C_c^0(\R^L)$. 
	Hence, $\btau_{-w}(T_p\widetilde{\Sigma}) = T_p\widetilde{\Sigma}$ for all $w\in T_pG\cdot p$, and $T_p\widetilde{\Sigma} = T_pG\cdot p \times P_p$ for some $(l-n+m)$-subspace $P_p\subset {\bf N}_pG\cdot p$. 

	Next, take $r\in (0, \inj(G\cdot p))$ so that the normal exponential map $\exp_{G\cdot p}^\perp$ is a diffeomorphism in $\tB^G_r(p)$. 
	Then by the {\em trivial} slice representation of principal orbits (\cite[Corollary 2.2.2]{berndt2016submanifolds}), we have a diffeomorphism $\phi: \tB^G_r(p)\to G\cdot p \times B_r([p])$ given by 
	\begin{equation}\label{Eq: local trivialization}
		\phi(q) = (g\cdot p, ~\pi(q)) \qquad \forall q\in \tB^G_r(p), 
	\end{equation}
	where $g\in G$ satisfies $(\exp_{G\cdot p}^\perp)^{-1}(q)\in  {\bf N}_{g\cdot p} G\cdot p $. 
	It now follows from the $G$-invariance of $\widetilde{\Sigma}$ that $\phi(\widetilde{\Sigma}\cap \tB^G_r(p)) =  G\cdot p \times (\pi(\widetilde{\Sigma})\cap B_r([p]) ) $, and thus
	\[ T_pG\cdot p \times T_{[p]}\pi(\widetilde{\Sigma}) = d\phi(T_p\widetilde{\Sigma} ) =  d\phi( T_pG\cdot p \times P_p) =  T_pG\cdot p \times d\pi (P_p)  .\]
	Therefore, the $(m-n+l)$-plane $d\pi(P_p)$ is the approximated tangent space of $\pi(\widetilde{\Sigma})$. 
		
	Finally, by the co-area formula, we have
	\begin{equation}\label{Eq: co-area for rectifiable set}
		\mH^m(\widetilde{\Sigma}) = \int_{\pi(\widetilde{\Sigma})} \vartheta([q])~d\mH^{l-n+m}([q]),
	\end{equation}
	and $\mH^{l-n+m}(\pi(\widetilde{\Sigma})) \leq C\mH^m(\widetilde{\Sigma})<\infty$. 
	Note also $\vartheta >0$ in $M^{prin}$. 
	Together, the approximated tangent space of $\pi(\widetilde{\Sigma})$ exists for $\mH^{l-n+m}$-a.e. $[p]\in \pi(\widetilde{\Sigma})$, which implies $\pi(\widetilde{\Sigma})$ is $(l-n+m)$-rectifiable. 
\end{proof}

By the above lemma, we can associate to every $\Omega\in \C^G(M)$ a Caccioppoli set $\pi(\Omega)\in \C(M/G)$ with perimeter estimates.

\begin{proposition}\label{Prop: reduce Caccioppoli $G$-sets}
	For any $G$-invariant subset $\Omega$ with $\Clos(\Omega)\subset M^{prin}$, we have $\Omega\in\C^G(M^{prin})$ if and only if $\pi(\Omega) \in \C(M^{prin}/G)$. 
	Additionally, the reduced boundary $\bd\pi(\Omega) = \pi(\bd\Omega)\in \mR_{l}(M^{prin}/G;\mZ_2)$ satisfies 
	\begin{equation}\label{Eq: reduce Caccioppoli sets boundary}
		\| \bd\Omega \|(U) = \int_{\pi(U)} \vartheta([q])  ~d\|\bd\pi(\Omega)\|([q])  
	\end{equation}
	for any relative open $G$-subset $U\subset M$ provided $\Omega\in\C^G(M^{prin})$ or $\pi(\Omega) \in \C(M^{prin}/G)$. 
\end{proposition} 
\begin{proof}
	Firstly, suppose $\Omega\in\C^G(M^{prin})$. 
	At any $p\in M^{prin}$, let $(d\pi)^*: T_{[p]}(M^{prin}/G)\to T_p M^{prin}$ be the adjoint of $d\pi$ satisfying $\langle u, (d\pi)^*(v)\rangle_{g_{_M}} = \langle d\pi(u), v\rangle_{g_{_{M/G}}}$ for any $u\in T_p M^{prin}, v\in T_{[p]}(M^{prin}/G)$. 
	Note $d\pi : {\bf N}_p(G\cdot p)\to T_{[p]}(M^{prin}/G)$ is an isometry, and thus $d\pi^*: T_{[p]}(M^{prin}/G) \to  {\bf N}_p(G\cdot p)$ is also an isometry. 
	
	Then for any $X\in \mathfrak{X}(M^{prin}/G)$, we can define $\tilde{X}\in \mathfrak{X}(M^{prin})$ by $\tilde{X}(p):= (d\pi)^*(X([p])) \in {\bf N}_pG\cdot p$ for all $p\in M$. 
	Note it follows from \cite[Theorem 14.3]{simon1983lectures} and the $G$-invariance of $\Omega$ (in particular (\ref{Eq: tangent space split})) that the measure-theoretic outer unit normal $\nu_{\Omega}$ of $\Omega$ is $G$-invariant and normal to every orbit. 
	Hence, using the co-area formula and the isometry $d\pi:{\bf N}_qG\cdot q \to T_{[q]}M/G$, 
	\begin{eqnarray*}
		\int_\Omega \Div \tilde{X} d\mH^{n+1} &= &  \int_{\bd\Omega} \tilde{X}\cdot\nu_{\bd\Omega} ~d\mH^{n} = \int_{\pi(\bd\Omega)} (\tilde{X}\cdot\nu_{\Omega})\big(\pi^{-1}([q])\big)\cdot \vartheta([q]) ~d\mH^{l}([q]) 
		\\&=& \int_{\pi(\bd\Omega)}  \big(\vartheta X \big) \cdot d\pi(\nu_{\Omega}) ~d\mH^{l} .
	\end{eqnarray*}
	On the other hand, by the co-area formula and the isometry $d\pi:{\bf N}_qG\cdot q \to T_{[q]}M/G$, 
	\begin{eqnarray*}
		\int_\Omega \Div \tilde{X} d\mH^{n+1} &= &  \int_{\Omega} \Div_{T_qG\cdot q}  \tilde{X} + \Div_{{\bf N}_qG\cdot q}  \tilde{X} ~d\mH^{n+1} (q)
		\\&=& \int_{\pi(\Omega)} \Big(\int_{G\cdot q}\Div_{T_qG\cdot q} \tilde{X} \Big) +  \big( \Div_{{\bf N}_qG\cdot q} \tilde{X}\big)\cdot \vartheta([q]) ~d\mH^{l+1} ([q])
		\\&=& \int_{\pi(\Omega)} \delta (G\cdot p)(\tilde{X}) + \big( \Div_{T_{[q]}M/G} X\big)\cdot \vartheta([q]) ~d\mH^{l+1} ([q])
		\\&=& \int_{\pi(\Omega)} \delta (G\cdot p)(\tilde{X}) + \Div_{T_{[q]}M/G} \big(\vartheta X\big) - X\big(\vartheta\big) ~d\mH^{l+1} 
		\\&=& \int_{\pi(\Omega)}  \Div \big(\vartheta X\big) ~d\mH^{l+1} , 
	\end{eqnarray*}
	where $\delta(G\cdot p)(\tilde{X})=X(\vartheta)$ is the first variation for the volume functional $\vartheta$ of $G\cdot p$. 
	Combining the above computations, we conclude $\pi(\Omega)\in\C(M/G)$, $\bd \pi(\Omega) = \pi(\bd\Omega)\in \mR_{l}(M^{prin}/G;\mZ_2)$ and $\nu_{\pi(\Omega)} = d\pi(\nu_{\Omega})$. 
	
	Next, suppose $\pi(\Omega)\in\C(M^{prin}/G)$. 
	Then for any $\tilde{X}\in\mathfrak{X}(M^{prin})$, consider the averaged $G$-invariant vector field $\tilde{X}_G:=\int_G (dg^{-1})\tilde{X} d\mu(g) \in \mathfrak{X}^G(M^{prin})$. 
	Hence, 
	\[ \int_\Omega \Div \tilde{X}_G = \int_G\Big(\int_\Omega \Div ((dg^{-1})\tilde{X} ) \Big) d\mu(g) = \int_G\Big(\int_{g\cdot \Omega} \Div \tilde{X}  \Big) d\mu(g) = \int_{\Omega} \Div \tilde{X} . \]
	In addition, denote by $\tilde{X}_G^\bot , \tilde{X}_G^\top$ the normal and the tangent parts of $\tilde{X}_G$ to every orbit respectively. 
	According to the previous computations, 
	\[ \int_\Omega \Div \tilde{X}_G^\bot = \int_{\pi(\Omega)}  \Div \big(\vartheta \cdot d\pi(\tilde{X}_G^\bot) \big) ~d\mH^{l+1} = \int_{\bd \pi(\Omega)}  \big(\vartheta d\pi(\tilde{X}_G^\bot) \big) \cdot \nu_{\pi(\Omega)} ~d\mH^{l}. \]
	Meanwhile, by the co-area formula and divergence theorem, we see
	\[\int_{\Omega}\Div \tilde{X}_G^\top = \int_{\Omega}\Div_{T(G\cdot p)} \tilde{X}_G^\top + \int_{\Omega}\Div_{{\bf N}(G\cdot p)} \tilde{X}_G^\top = \int_{\Omega}\Div_{{\bf N}(G\cdot p)} \tilde{X}_G^\top .\] 
	Note $\Div_{{\bf N}(G\cdot p)} \tilde{X}_G^\top (p)= \langle H_{S_{p}}, \tilde{X}_G^\top \rangle$, where $S_p:=\exp_{G\cdot p}^\perp(\mathbb{B}^{l+1}_{\inj(G\cdot p)}(0)\cap {\bf N}_p(G\cdot p))$ is the slice of $G\cdot p$ at $p$, and $H_{S_p}\in (T_p(S_p))^\perp=T_p(G\cdot p)$ is the mean curvature vector of $S_p$ at $p$.
	Since $\exp^\perp(p,v):=\exp_p(v)$ is a smooth map on $\{(p,v): p\in M, v\in {\bf N}_p(G\cdot p), |v|<\inj(G\cdot p)\}$, we know $|H_{S_p}|^2$ is smooth on $M=M^{prin}$ and $|H_{S_p}|\leq C$ for some constant $C=C(M,G)>0$, which implies $|\Div_{{\bf N}(G\cdot p)} \tilde{X}_G^\top| \leq C |\tilde{X}_G^\top|$. 
	Therefore, we have $1_{\Omega}\in BV(M)$ and $\Omega\in \C^G(M)$.

	Finally, (\ref{Eq: reduce Caccioppoli sets boundary}) follows from Lemma \ref{Lem: reduce rectifiable G-set} and the co-area formula (see (\ref{Eq: co-area for rectifiable set})). 
\end{proof}

Using the above proposition, we obtain the following bi-Lipschitz homeomorphism between $\Z_{n}^G(M,\partial M; \M; \mZ_2) $ and $ \Z_{l}(M/G, \partial M/G; \M; \mZ_2)$ provided $M=M^{prin}$. 
\begin{theorem}\label{Thm: homeomorphism}
	Suppose $M=M^{prin}$. 
	Then there is a bi-Lipschitz homeomorphism 
	\[\Xi : \Z_{n}^G(M,\partial M; \M; \mZ_2) \to \Z_{l}(M/G, \partial M/G; \M; \mZ_2), \]
	so that ${\rm Lip}(\Xi)\leq 1/\inf_{p\in M}\vartheta(G\cdot p)$ and ${\rm Lip}(\Xi^{-1})\leq \sup_{p\in M}\vartheta(G\cdot p)$. 
	Additionally, for any $\tilde{\tau}\in \Z_{n}^G(M,\partial M; \mZ_2)$ and relative open $G$-subset $U\subset M$, 
	\begin{equation}\label{Eq: reduced mass}
		\|T_{\tilde{\tau}}\|(U) = \|T_{\Xi(\tilde{\tau})}\|\llcorner \vartheta~(\pi(U)) := \int_{\pi(U)} \vartheta([q])~d\|T_{\Xi(\tilde{\tau})}\|([q]),
	\end{equation}
	where $\vartheta$ is the volume function of orbits defined in (\ref{Eq: orbit volume}), $T_{\tilde{\tau}}$ and $T_{\Xi(\tilde{\tau})}$ are the canonical representative of $\tilde{\tau}$ and $\Xi(\tilde{\tau})$ respectively. 
\end{theorem}
\begin{proof}
	By definitions, for any $\tilde{\tau}\in \Z_{n}^G(M,\partial M;\mZ_2)$, there exists $\widetilde{\Omega}\in \C^G(M)$ so that $\tilde{\tau} = [\partial \widetilde{\Omega}] = [\partial (M-\widetilde{\Omega})]$. 
	By Proposition \ref{Prop: reduce Caccioppoli $G$-sets}, we have $\tau:= [\partial \pi(\widetilde{\Omega})] = [\partial  \pi(M-\widetilde{\Omega})] \in \Z_{l}(M/G, \partial M/G; \mZ_2)$ and $\M(\tau) = \|\partial \pi(\widetilde{\Omega})\|(\interior(M/G)) \leq C_1 \|\partial \widetilde{\Omega}\|(\interior(M)) = C_1\M(\tilde{\tau})$, where $C_1 = 1/\inf_{p\in M}\vartheta(G\cdot p)$. 
	Hence, the map $\Xi: \tilde{\tau} \mapsto \tau$ is Lipschitz continuous with ${\rm Lip}(\Xi)\leq C_1$, and (\ref{Eq: reduced mass}) follows from (\ref{Eq: reduce Caccioppoli sets boundary}). 
	
	On the other hand, for any $\tau\in \Z_{l}(M/G, \partial M/G; \mZ_2)$, we take $\Omega\in \C(M/G)$ so that $\tau = [\partial \Omega] = [\partial (M/G - \Omega)]$. 
	By Proposition \ref{Prop: reduce Caccioppoli $G$-sets}, the map $\Lambda: \tau \mapsto \tilde{\tau}: = [\partial  \pi^{-1}(\Omega)]$ is a Lipschitz continuous map with ${\rm Lip}(\Lambda)\leq \sup_{p\in M}\vartheta(G\cdot p)$, and $ \Lambda\circ\Xi = id$, $\Xi\circ\Lambda=id$. 
	Therefore, $\Xi$ is a bi-Lipschitz homeomorphism. 
\end{proof}


Recall $\vartheta : M/G\to \R^+$ given by (\ref{Eq: orbit volume}) is the volume function of orbits. 
Then, for a compact Riemannian manifold $(M, g_{_M})$ with $M=M^{prin}$, we can rescale the induced metric $g_{_{M/G}}$ of the orbit space $M/G$ as 
\begin{equation}\label{Eq: rescaled metric}
	\tilde{g}_{_{M/G}}([p]) := (\vartheta([p]))^{\frac{2}{l}}\cdot g_{_{M/G}}([p]).
\end{equation}
Thus, for any $G$-invariant $n$-rectifiable set $\tilde{\Gamma}$, we have the Jacobian $J_\pi^{\tilde{\Gamma}*} (p)= \vartheta(\pi(q))$ in the co-area formula under the rescaled metric $\tilde{g}_{_{M/G}}$.  
As in (\ref{Eq: co-area for rectifiable set})(\ref{Eq: reduce Caccioppoli sets boundary}), we obtain 
\begin{equation}\label{Eq: reduce caccioppoli set under rescaled metric}
	\|\partial \Omega \|_{g_{_M}}(U) = \|\partial \pi(\Omega)\|_{\tilde{g}_{_{M/G}}}(\pi(U)), 
\end{equation}
where the subscripts are used to emphasize the metrics. 
Therefore, the map 
\begin{equation}\label{Eq: homeomorphism is isometry}
	\Xi : \Z_{n}^G(M,\partial M; \M_{g_{_M}}; \mZ_2) \to \Z_{l-1}(M/G, \partial M/G; \M_{\tilde{g}_{_{M/G}}}; \mZ_2), \quad \Xi([\bd\Omega]) := [\bd\pi(\Omega)] ,
\end{equation}
given in Theorem \ref{Thm: homeomorphism} is an isometry under the rescaled metric $\tilde{g}_{_{M/G}}$, where $\M_{g_{_M}} $ and $\M_{\tilde{g}_{_{M/G}}}$ are the mass norm with respect to the metric $g_{_M}$ and $\tilde{g}_{_{M/G}}$ respectively. 

Now, for any $\M_{g_{_M}}$-continuous $(G,p)$-sweepout $\Phi\in \mathcal{P}^G_p(M,g_{_M})$, since $\Xi$ is an isometry, we have $\Xi\circ\Phi \in \mathcal{P}_p(M/G, \tilde{g}_{_{M/G}})$ is an $\M_{\tilde{g}_{_{M/G}}}$-continous $p$-sweepout in $M/G$ such that $\M_{\tilde{g}_{_{M/G}}}(\Xi\circ\Phi(x)) = \M_{g_{_M}}(\Phi(x))$, $\forall x\in{\rm dmn}(\Phi)$. 
Hence, we have the following theorem:
\begin{theorem}\label{Thm: weyl law for single orbit type}
	Let $(M, g_{_M})$ be a connected compact Riemannian manifold with smooth and possibly empty boundary. 
	Let $G$ be a compact Lie group acting by isometries on $M$ with ${\rm Cohom}(G)=l+1\geq 2$. 
	Suppose there only exists a single orbit type on $M$, i.e. $M=M^{prin}$. 
	Then $\omega_p^G(M,g_{_M}) = \omega_p(M/G, \tilde{g}_{_{M/G}})$ and 
	\begin{eqnarray}\label{Eq: Weyl law for single orbit type}
		\lim_{p\to\infty} p^{-\frac{1}{l+1}}\omega_p^G(M,g_{_M}) 
		&=& a(l)\cdot {\rm Vol}(M/G, \tilde{g}_{_{M/G}})^{\frac{l}{l+1}} 
		\\
		&=& a(l)\cdot \Big( \int_M (\vartheta(G\cdot q))^{\frac{1}{l}}d\mH^{n+1}(q) \Big)^{\frac{l}{l+1}}, \nonumber
	\end{eqnarray}
	where $\tilde{g}_{_{M/G}}$ is given by (\ref{Eq: rescaled metric}), and $a(l)$ is a constant depending only on $l$ (cf. \cite[Theorem 1.1]{liokumovich2018weyl}). 
\end{theorem}
\begin{proof}
	Since $\Z_{n}^G(M,\partial M; \M_{g_{_M}}; \mZ_2)$ is isometric to $\Z_{l-1}(M/G, \partial M/G; \M_{\tilde{g}_{_{M/G}}}; \mZ_2)$, we immediately have $\omega_p^G(M,g_{_M}) = \omega_p(M/G, \tilde{g}_{_{M/G}})$. 
	Then, (\ref{Eq: Weyl law for single orbit type}) follows from Remark \ref{Rem: equivalence between p-widths definition}, Weyl law \cite[Theorem 1.1]{liokumovich2018weyl} in $(M/G, \tilde{g}_{_{M/G}})$, (\ref{Eq: rescaled metric})(\ref{Eq: reduce caccioppoli set under rescaled metric}) and the co-area formula. 
\end{proof}

In the above theorem, the smoothness assumption of $\partial M$ can be weakened. 
Specifically, we have the following corollary. 
\begin{corollary}\label{Cor: weyl law for single type}
	Let $(\widetilde{M}, g_{_{\widetilde{M}}})$ be a smooth Riemannian manifold with a compact Lie group $G$ acting isometrically so that ${\rm Cohom}(G)=l+1\geq 2$.
	Suppose $M\subset\subset \widetilde{M}^{prin}$ is a compact $G$-invariant domain  whose boundary $\partial M$ is the union of a finite number of smooth embedded $G$-hypersurfaces $\{\widetilde{\Sigma}_k\}_{k=1}^K$ that meet each other transversally. 
	Then (\ref{Eq: Weyl law for single orbit type}) is also valid. 
\end{corollary}
\begin{proof}
	By the proof in \cite[Theorem 4.1, 4.2]{liokumovich2018weyl}, Weyl law in compact manifolds (\cite[Theorem 4.2]{liokumovich2018weyl}) also holds for $M/G$ with piecewise smooth boundary. 
	Then the arguments can be taken verbatim from above. 
\end{proof}

\subsection{Weyl law in general $G$-manifolds}\label{Subsec: general manifold}

In this subsection, we consider a general connected compact Riemannian manifold $(M^{n+1}, g_{_M})$ with smooth (possibly empty) boundary, and a compact Lie group $G$ acting by isometries on $M$ with cohomogeneity ${\rm Cohom}(G)=l+1\geq 2$. 
Note the union of principal orbits $M^{prin}$ forms an open dense subset of $M$. 
In particular, $\mH^{n+1}(M\setminus M^{prin} ) = 0 $.
Additionally, since $M$ is compact, the number orbit types in $M$ is finite, which implies 
\[\mH^n(M \setminus M^{prin} ) \leq C_{M,G} < \infty. \]
Denote by $g_{_{M/G}}$ the induced Riemannian metric on $M^{prin}/G$ so that $\pi:M^{prin}\to M^{prin}/G$ is a Riemannian submersion. 
We also define $\tilde{g}_{_{M/G}}$ to be the rescaled the Riemannian metric on $M^{prin}/G$ as in (\ref{Eq: rescaled metric}). 
Note $\tilde{g}_{_{M/G}}([x])$ may tend to $0$ as $[x]\to \pi(M\setminus M^{prin})$.

To begin with, let us show the following result using Theorem \ref{Thm: weyl law for single orbit type} and the Lusternik-Schnirelmann Inequality in \cite[Theorem 3.1]{liokumovich2018weyl}. 
\begin{theorem}\label{Thm: lower bound weyl law}
	$\liminf_{p\to\infty}p^{-\frac{1}{l+1}}\omega_p^G(M, g_{_M}) \geq  a(l){\rm Vol}(M^{prin}/G, \tilde{g}_{_{M/G}})^{\frac{l}{l+1}}$. 
\end{theorem}
\begin{proof}
	The proof is similar to \cite[Theorem 4.1]{liokumovich2018weyl}. 
	Firstly, for any $\epsilon >0$, choose a collection of $G$-tubes $\{\tB_{r_k}^G(y_k)\}_{k=1}^K \subset M^{prin}\setminus\bd M$ so that
	\begin{itemize}
		\item  $\sum_{k=1}^K {\rm Vol}(\pi(\tB_{r_k}^G(y_k)), ~\tilde{g}_{_{M/G}}) \geq (1+ \epsilon )^{-1} {\rm Vol}(M^{prin}/G, ~\tilde{g}_{_{M/G}}) $;
		\item $\tB_{r_k}^G(y_k)\cap \tB_{r_j}^G(y_j) = \emptyset$ for all $k\neq j \in \{1,\dots, K\}$;
		\item $r_k < \frac{1}{4}\min\{\inj(G\cdot y_k), \dist_{M}(G\cdot y_k, \bd M)\}$, for all $k\in \{1,\dots, K\}$. 
	\end{itemize}
	For simplicity, denote by $B^G_k := \tB^G_{r_k}(y_k)$ and $B_k = \pi(\tB^G_{r_k}(y_k))$. 
	
	\begin{claim}\label{Claim: sweepout subset}
		Given $1\leq k\leq K$ and an $\M$-continuous $\Phi\in \mathcal{P}_p^G(M)$ with $X={\rm dmn}(\Phi)$, then 
		\begin{itemize}
			\item[(a)] for any $\epsilon_1>0$, there is an $\M$-continuous $(G,p)$-sweepout $\hat{\Phi}_k: X\to \Z_n^G(B^G_k, \partial B^G_k; \mZ_2)$ of $B^G_k$ so that $\M(\hat{\Phi}_k(x)) \leq 	(1+ \epsilon_1)^k\M(\Phi(x)\llcorner B^G_k )$ for all $x\in X$; 
			\item[(b)] for any integer $0\leq q \leq p$ and $\epsilon_2>0$, the open set 
				\[ \left\{x\in X : \M(\Phi(x) \llcorner B^G_k ) < \omega_q^G(B^G_k, g_{_M}) - \epsilon_2 \right\} \]
				is contained in an open set $U_k\subset X$ so that $(\iota_k^*\lambda)^q = (\iota_k^*\Phi^*\bar{\lambda})^q = 0 \in H^q(U_k;\mZ_2)$, where $\iota_k:U_k\to X$ is the inclusion map, $\bar{\lambda}$ is the generator of $H^1(\Z_n^G(M,\partial M;\mZ_2 ); \mZ_2 )$, and $\lambda := \Phi^*\bar{\lambda}\in H^1(X;\mZ_2)$.  
		\end{itemize} 
	\end{claim}
	\begin{proof}[Proof of Claim \ref{Claim: sweepout subset}]
		The proof is parallel to the one in \cite[Lemma 2.15]{liokumovich2018weyl}, and we only point out some modifications. 
		For each $k\in \{1,\dots, K\}$, let $\eta_k\geq 0$ be a $G$-invariant smooth cut-off function compactly supported in $\tB_{4r_k}^G(y_k)\setminus \tB_{r_k/4}^G(y_k)$ so that $\eta_k = 1 $ in $ \tB_{2r_k}^G(y_k) \setminus \tB_{r_k/2}^G(y_k)$. 
		Denote by $\{F^k_t\}$ the equivariant diffeomorphisms generated by $Y_k := -\eta_k\cdot \nabla \dist_{M}(G\cdot y_k,\cdot) \in \mathfrak{X}^G(M)$. 
		Note $Y_k$ is the inward unit normal of $\tB^G_r(y_k)$ for $r$ close to $r_k$. 
		Hence, we can replace $R$, $Y$, $F_t$, $u$ in the constructions of \cite[Lemma 2.15]{liokumovich2018weyl} by $B^G_k$, $Y_k$, $F^k_t$, $\dist_{M}(G\cdot p_k,\cdot)$. 
		Additionally, we also use \cite[Theorem 4.13]{wang2023min} and Proposition \ref{Prop: homotopy between flat close maps} in place of \cite[Theorem 2.11, Proposition 2.12]{liokumovich2018weyl} used on \cite[Page 945-946]{liokumovich2018weyl}. 
		Then the rest of the proof can be taken almost verbatim from \cite[Lemma 2.15]{liokumovich2018weyl}. 
	\end{proof}
	
	Now, for any integer $p\in\N$ and $k\in \{1,\dots, K\}$, define 
	\[p_k:= \left\lfloor p\cdot \frac{{\rm Vol}(B_k, \tilde{g}_{_{M/G}})}{{\rm Vol}(M^{prin}/G, \tilde{g}_{_{M/G}})} \right\rfloor \in \N ,\]
	where $\lfloor x\rfloor$ is the integer part of $x\in\R$. 
	Clearly, $0\leq p_k \leq p$, and $\bar{p}:=\sum_{k=1}^K p_k\leq p$. 
	
	Next, for any $\epsilon_2 >0$ and $\M$-continuous $\Phi\in \mathcal{P}_p^G(M)$ with $X={\rm dmn}(\Phi)$, take $q= p_k$ and consider the open set $U_k \stackrel{\iota_k}{\hookrightarrow} X= {\rm dmn}(\Phi)$ given by Claim \ref{Claim: sweepout subset}(b).  
	Combining $(\iota_k^*\lambda)^{p_k} = 0$ in Claim \ref{Claim: sweepout subset}(b) with the exact sequence
	\[H^{p_k}(X, U_k; \mZ_2)  \xrightarrow{j^*} H^{p_k}(X; \mZ_2)  \xrightarrow{\iota_k^*} H^{p_k}(U_k; \mZ_2), \]
	we can find $\lambda_k\in H^{p_k}(X, U_k; \mZ_2) $ with $j^{*}(\lambda_k) = \lambda^{p_k}$. 
	Since $\Phi$ is a $(G,p)$-sweepout and $\bar{p}=\sum_{k=1}^K p_k\leq p$, we have $\lambda^p = (\Phi^*\bar{\lambda})^p \neq 0 $ and 
	\[j^*(\lambda_1 )\smile\cdots\smile j^*(\lambda_K ) = \lambda^{\sum_{k=1}^K p_k} = \lambda^{\bar{p}} \neq 0 \in H^{\bar{p}}(X;\mZ_2). \]
	If $X=\cup_{k=1}^K U_k$, then $\lambda_1\smile\cdots\smile \lambda_K \in H^{\bar{p}}(X, \cup_{k=1}^K U_k ;\mZ_2) = H^{\bar{p}}(X, X;\mZ_2) = 0$, which contradicts $0\neq \lambda^{\bar{p}} = j^*(\lambda_1\smile\cdots\smile \lambda_K)$. 
	Therefore, there exists $x\in X\setminus \cup_{k=1}^K U_k$, which satisfies $\M(\Phi(x)) \geq \sum_{k=1}^K \M(\Phi(x)\llcorner B^G_k )\geq \sum_{k=1}^K\omega_{p_k}^G(B^G_k, g_{_M}) - K\cdot \epsilon_2 $. 
	Since $\epsilon_2>0$ is arbitrary and $\Phi$ is an arbitrary $\M$-continuous $(G,p)$-sweepout of $M$, we conclude 
	\[\omega^G_p(M, g_{_M}) \geq \sum_{k=1}^K\omega_{p_k}^G(B^G_k, g_{_M})\] 
	by Proposition \ref{Prop: p-width using M-continuous map}. 
	Multiplying $p^{-\frac{1}{l+1}}$ on both sides, 
	\begin{eqnarray*}
		p^{-\frac{1}{l+1}} \omega^G_p(M, g_{_M}) &\geq & \sum_{k=1}^K p^{-\frac{1}{l+1}} \omega_{p_k}^G(B^G_k, g_{_M}) =  \sum_{k=1}^K  \left(\frac{p_k}{p} \right)^{\frac{1}{l+1}} \cdot p_k^{-\frac{1}{l+1}} \omega_{p_k}^G(B^G_k, g_{_M})
		\\&\geq& \sum_{k=1}^K  \left( \frac{{\rm Vol}(B_k, \tilde{g}_{_{M/G}})}{{\rm Vol}(M^{prin}/G, \tilde{g}_{_{M/G}})} - \frac{1}{p}  \right)^{\frac{1}{l+1}} \cdot p_k^{-\frac{1}{l+1}} \omega_{p_k}^G(B^G_k, g_{_M}).
	\end{eqnarray*}
	Noting $B^G_k\subset\subset M^{prin}$, it follows from Theorem \ref{Thm: weyl law for single orbit type} and the choice of $\{B^G_k\}_{k=1}^K$ that 
	\begin{eqnarray*}
		\liminf_{p\to\infty} p^{-\frac{1}{l+1}} \omega^G_p(M, g_{_M}) &\geq &\sum_{k=1}^K \left( \frac{{\rm Vol}(B_k, \tilde{g}_{_{M/G}})}{{\rm Vol}(M^{prin}/G, \tilde{g}_{_{M/G}})} \right)^{\frac{1}{l+1}} \cdot a(l) {\rm Vol}(B_k, \tilde{g}_{_{M/G}})^{\frac{l}{l+1}}
		\\&=& a(l) {\rm Vol}(M^{prin}/G, \tilde{g}_{_{M/G}})^{-\frac{1}{l+1}} \sum_{k=1}^K {\rm Vol}(B_k, \tilde{g}_{_{M/G}})
		\\&\geq & a(l) {\rm Vol}(M^{prin}/G, \tilde{g}_{_{M/G}})^{\frac{l}{l+1}}\cdot (1+\epsilon )^{-1}. 
	\end{eqnarray*}
	Finally, the desired estimate follows from taking $\epsilon\to 0$. 
\end{proof}

\begin{remark}
	We mention that the result in Claim \ref{Claim: sweepout subset}(a) can be adapted to more general cases and show $\omega^G_p(A)\leq \omega^G_p(B)$ for any $G$-invariant domains $A\subset B\subset M$ with piecewise smooth boundary. 
\end{remark}

Next, we shall make use of the density of $M^{prin}$ to show the Weyl Law of the equivariant volume spectrum $\{\omega^G_p(M,g_{_M})\}_{p\in\mZ_+}$. 

Recall $(\widetilde{M}, g_{_{\widetilde{M}}})$ is a closed Riemannian manifold with $G$ acts by isometries of ${\rm Cohom}(G)=l+1\geq 2$ so that $M$ is equivariantly and isometrically embedded in $\widetilde{M}$ as a smooth compact domain (see Section \ref{Subsec: notations in manifolds}). 
Let $g_{_{\widetilde{M}/G}}$ be the induced Riemannian metric on $\widetilde{M}^{prin}/G$ so that $\pi:\widetilde{M}^{prin}\to \widetilde{M}^{prin}/G$ is a Riemannian submersion. 
Then, we also define the rescaled Riemannian metric $\tilde{g}_{_{\widetilde{M}/G}}$ on $\widetilde{M}^{prin}/G$ as in (\ref{Eq: rescaled metric}) so that $\tilde{g}_{_{\widetilde{M}/G}}= \tilde{g}_{_{M/G}}$ in $M^{prin}/G$. 
Additionally, let $\{\widetilde{M}_{(H_i)}\}_{i=1}^I$, $I\in\N$, be the distinct non-principal orbit type strata in $\widetilde{M}\setminus\widetilde{M}^{prin}$. 

\begin{theorem}
	Suppose $(M^{n+1}, g_{_M})$ is a connected compact Riemannian manifold with smooth (possibly empty) boundary, and $G$ is a compact Lie group acting by isometries on $M$ with cohomogeneity ${\rm Cohom}(G)=l+1\geq 2$. 
	Then, 
	\begin{eqnarray*}
		\lim_{p\to\infty}p^{-\frac{1}{l+1}}\omega_p^G(M, g_{_M}) &=&  a(l){\rm Vol}(M^{prin}/G, \tilde{g}_{_{M/G}})^{\frac{l}{l+1}}
		\\&=& a(l) \left( \int_{M} (\vartheta(G\cdot q))^{\frac{1}{l}}d\mH^{n+1}(q) \right)^{\frac{l}{l+1}},
	\end{eqnarray*}
	where $\tilde{g}_{_{M/G}}$ is the rescaled metric defined as in (\ref{Eq: rescaled metric}), $\vartheta(G\cdot q):=\mH^{n-l}(G\cdot q)$ is the orbits volume function, and $a(l)$ is a constant depending only on $l$ given in \cite[Theorem 1.1]{liokumovich2018weyl}. 
\end{theorem}
\begin{proof}
	By Theorem \ref{Thm: lower bound weyl law}, it is sufficient to show	
	\begin{equation}\label{Eq: upper bound weyl law}
		\limsup_{p\to\infty}p^{-\frac{1}{l+1}}\omega_p^G(M, g_{_M}) \leq  a(l){\rm Vol}(M^{prin}/G, \tilde{g}_{_{M/G}})^{\frac{l}{l+1}}.
	\end{equation}
	Then the desired equality follows from the co-area formula and $\mH^{n+1}(M\setminus M^{prin}) = 0$. 

	Fix any $\epsilon>0$. 
	Let $\widetilde{M}$, $g_{_{\widetilde{M}}}$, $\tilde{g}_{_{\widetilde{M}/G}}$, and $\{\widetilde{M}_{(H_i)}\}_{i=1}^I$ be defined as before so that $M$ is a $G$-invariant compact domain in $\widetilde{M}$ with smooth boundary $\partial M$. 
	
	{\bf Step 1.} 
	Using the retraction map of the tubular neighborhood $\tB_r(\partial M)$, there exist $t_0>0$ sufficiently small and a smooth $G$-equivariant map $F_0: \widetilde{M}\to \widetilde{M}$ such that  
	\begin{itemize}
		\item[(a)] ${\rm Lip}(F_0)\leq 1+\epsilon$; 
		\item[(b)] $F_0( \Clos(\tB_{t_0}(\partial M)) ) = \partial M$ and $F_0(\widetilde{M} \setminus \tB_{t_0}(\partial M)) = \widetilde{M}$; 
		\item[(c)] $F_0(M \setminus \tB_{t_0}(\partial M)) ) = F_0(M \cup \tB_{t_0}(\partial M)) ) = M$. 
	\end{itemize}
	Indeed, let $\exp_{\partial M}^\bot$ be the normal exponential map of $\partial M$ in $(\widetilde{M}, g_{_{\widetilde{M}}})$. 
	Then for $R>0$ sufficiently small, $\exp_{\partial M}^\bot$ is a $(1+\epsilon)^{\frac{1}{8}}$-bilipschitz $G$-equivariant diffeomorphism in $\tB_{2R}(\partial M)$. 
	Take $0< t_0 < (1 - (1+ \epsilon)^{-\frac{1}{4}}) R$ and choose a smooth function $h: [0, 2R]\to [0, 2R]$ so that $h\llcorner [0, t_0] =0$, $h(t)=t$ for $t\in [R, 2R]$, and $0< h'(t) \leq (1+\epsilon )^{\frac{1}{4}} \frac{R}{R - t_0} < (1+\epsilon )^{\frac{1}{2}}$ for $t\in (t_0, 2R]$. 
	Thus, $h$ maps $[t_0, 2R]$ to $[0, 2R]$ with $h(t)\leq (1+\epsilon)^{\frac{1}{2}} t$. 
	We can then define the $G$-equivariant smooth map $F_0 : \widetilde{M} \to \widetilde{M}$ with ${\rm Lip}(F_0)\leq 1+\epsilon$ by
	\begin{equation}\label{Eq: compresing}
		 F_0(p)	:=  \begin{cases}
							p ,& {\rm if}~~ p \in \widetilde{M} \setminus \tB_{2R}(\partial M) ; \\
				 			\exp_{\partial M}^\bot\left( \frac{h(|v|)}{|v|} v \right) ,& {\rm if}~~ p = \exp_{\partial M}^\bot(v) \in \tB_{2R}(\partial M).
						\end{cases}
	\end{equation}
	One can easily verify that $F_0$ satisfies (a)-(c) and $F_0\llcorner \tM\setminus\Clos(\tB_{t_0}(\bd M) )$ is injective.  
	
	{\bf Step 2.} {\it Remove a neighborhood of the `most singular' orbit type stratum.}
	
	Denote by $M_1 :=M$. 
	Then we take the `most singular' orbit type stratum among $\{\widetilde{M}_{(H_i)} \}_{i=1}^I$, say $\widetilde{M}_{(H_{1})}$, in the sense that there is no $i\geq 2$ with $(H_1)\leq (H_i)$. 
	Note that $N_1 := \Clos(\widetilde{M}_{(H_{1})})=\widetilde{M}_{(H_{1})}$ is a closed $G$-invariant submanifold of $\widetilde{M}$. 
	
	Hence, considering the tubular $G$-neighborhood $\tB_r(N_1)$ in $\widetilde{M}$, there exists a positive number $t_1>0$ sufficiently small and a $G$-equivariant map $F_1: \widetilde{M}\to \widetilde{M}$ such that ${\rm Lip}(F_1)\leq 1+\epsilon$, $F_1( \Clos(\tB_{t_1}(N_1)) ) = N_1$, $F_1(\widetilde{M} \setminus \tB_{t_1}(N_1)) = \widetilde{M}$, $F_1\llcorner \widetilde{M} \setminus \Clos(\tB_{t_1}(N_1))$ is injective, and
	\[ M \setminus \tB_{t_0}(\partial M) \subset  F_1\Big( M_{1} \setminus \big( \tB_{t_1}(N_1) \cup \tB_{t_1}(\partial M_1) \big) \Big) \subset  F_1\big( M_{1} \cup \tB_{t_1}(\partial M_1) \big)  \subset M \cup \tB_{t_0}(\partial M). \]
	The construction of $F_1$ is similar to $F_0$ in (\ref{Eq: compresing}). 
	Note for $t_1\in (0,t_0)$ small enough, $F_1\big( M_{1} \setminus \big( \tB_{t_1}(N_1) \cup \tB_{t_1}(\partial M_1) \big) \big)$ and $ F_1\big( M_{1} \cup \tB_{t_1}(\partial M_1) \big)$ are sufficiently close to $M_{1}=M$. 
	Hence, combining this with the fact that $M \setminus \tB_{t_0}(\partial M)\subset \subset M\subset \subset M \cup \tB_{t_0}(\partial M)$, we have the above inclusion relationship for $t_1>0$ sufficiently small. 
	
	Additionally, by Sard's theorem, we can choose $r_1\in(0, t_1)$ such that $\partial \tB_{r_1}(N_1)$ intersects $\partial M_1$ transversally. 
	Then we define 
	\[M_2 := M_1 \setminus \tB_{r_1}(N_1) = M\setminus \tB_{r_1}(N_1)  ,\]
	which is a compact $G$-domain in $\widetilde{M}$ with piecewise smooth boundary $\partial \big(M\setminus \tB_{r_1}(N_1) \big)$ satisfying $M_2\cap N_1 = \emptyset$. 
	Moreover, we have
	\begin{equation}\label{Eq: compare sets}
		M_1 \setminus \big( \tB_{t_1}(N_1) \cup \tB_{t_1}(\partial M_1) \big)  ~\subset \subset ~ M_2 ~\subset \subset ~ M_1 \cup \tB_{t_1}(\partial M_1),
	\end{equation}
	and $\tB_{r_1/2}(M_2)\cap \tM_{(H_1)} = \emptyset$, by the constructions. 
	
	{\bf Step 3.} {\it Remove a neighborhood of the `second most singular' orbit type stratum.}
	
	Next, take the `most singular' orbit type stratum among $\{\widetilde{M}_{(H_{i})}\}_{i=2}^I$, say $\widetilde{M}_{(H_{2})}$, in the sense that there is no $i\geq 3$ with $(H_2)\leq (H_i)$. 
	Hence, either
	\begin{itemize}
		\item[(1)] $\emptyset= \Clos(\widetilde{M}_{(H_{2})})\setminus \widetilde{M}_{(H_{2})}$ and $\widetilde{M}_{(H_{2})}$ is a closed $G$-submanifold; or
		\item[(2)] $\emptyset\neq \Clos(\widetilde{M}_{(H_{2})})\setminus \widetilde{M}_{(H_{2})} \subset \widetilde{M}_{(H_{1})}$ and $\widetilde{M}_{(H_{2})}$ is a $G$-submanifold in $\tM\setminus \tM_{(H_1)}$.
	\end{itemize}
	
	In case (1), $N_2 := \widetilde{M}_{(H_{2})}$ is a closed $G$-submanifold of $\widetilde{M}$. 
	As in {\bf Step 2}, using the retraction in the tubular $G$-neighborhood $\tB_r(N_2)$, we have positive numbers $r_1/4>t_2>r_2>0$ sufficiently small and a $G$-equivariant map $F_2: \widetilde{M}\to \widetilde{M}$ such that ${\rm Lip}(F_2)\leq 1+\epsilon$, $F_2( \Clos(\tB_{t_2}(N_2)) ) = N_2$, $F_2(\widetilde{M} \setminus \tB_{t_2}(N_2)) = \widetilde{M}$, $F_2\llcorner \widetilde{M} \setminus \Clos(\tB_{t_2}(N_2)))$ is injective, 
	\begin{eqnarray*}
		M_1 \setminus \big( \tB_{t_1}(N_1) \cup \tB_{t_1}(\partial M_1) \big) &\subset & F_2\big( M_{2} \setminus \big( \tB_{t_2}(N_2) \cup \tB_{t_2}(\partial M_2) \big) \big) 
		\\&\subset & F_2\big( M_{2} \cup \tB_{t_2}(\partial M_2) \big)  ~\subset ~ M_1 \cup \tB_{t_1}(\partial M_1) ,
	\end{eqnarray*}
	and $\partial \tB_{r_2}(N_2)$ intersects $\partial M_2$ transversally. 
	We mention that the above inclusions are valid for $t_2>0$ small enough because of \eqref{Eq: compare sets} and the fact that $F_2\big( M_{2} \setminus \big( \tB_{t_2}(N_2) \cup \tB_{t_2}(\partial M_2) \big) \big) $ as well as $F_2\big( M_{2} \cup \tB_{t_2}(\partial M_2) \big)$ is sufficiently close to $M_2$ (as $t_2\to 0_+$). 
	
	In case (2), since $\tB_{r_1/2}(M_2)\cap \tM_{(H_1)}=\emptyset$, we can take $N_2\subset\subset \tM_{(H_2)}$ to be a $G$-domain of $\tM_{(H_2)}$ with smooth boundary so that $\tM_{(H_2)}\cap M_2\subset\subset N_2 \subset\subset \tM_{(H_2)}\cap \tB_{r_1/2}(M_2)$. 
	Then for $R\in (0, r_1/4)$ small enough, $\exp_{N_2}^\perp$ and $\exp_{\bd N_2}^\perp$ are $(1+\epsilon)^{\frac{1}{8}}$-bilipschitz $G$-equivariant diffeomorphisms in the tubular neighborhoods $\widetilde T_{2R}(N_2)$ and $\widetilde T_{2R}(\bd N_2)=\tB_{2R}(\bd N_2)$ respectively. 
	By shrinking $R>0$ even smaller, we can make $\tB_{2R}(N_2)\cap M_2 \subset \widetilde T_{2R}(N_2)$ since $N_2\cap M_2\subset\subset N_2$. 
	Let $0< t_2 < (1 - (1+ \epsilon)^{-\frac{1}{4}}) R< r_1/4$ and $h$ be given as in {\bf Step 1} with $t_2$ in place of $t_0$. 
	We can similarly define the $G$-equivariant compressing map $F_2: \tM\to \tM$ so that $F_2=id $ in $\tM\setminus \tB_{2R}(N_2)$, and $F\llcorner \tB_{2R}(N_2)$ is a $\left(\frac{h\circ \dist_{N_2}}{\dist_{N_2}}\right)$-scaling along the shortest geodesic to $N_2$. 
	Namely, $F_2(p):=\exp_{N_2}^\perp(\frac{h(|v|)}{|v|} v )$ for $p=\exp_{N_2}^\perp(v)\in \widetilde T_{2R}(N_2)$, and $F_2(p):=\exp_{\bd N_2}^\perp(\frac{h(|v|)}{|v|} v )$ for $p=\exp_{\bd N_2}^\perp(v)\in \widetilde T_{2R}(\bd N_2)\setminus \widetilde T_{2R}(N_2)$. 
	After shrinking $t_2>0$ even smaller, $F_2$ and $t_2$ satisfy the same properties as in case (1). 
	Moreover, since $\tB_{2R}(N_2)\cap M_2 \subset \widetilde T_{2R}(N_2)$, we can take $r_2\in (0,t_2)$ so that $\bd \tB_{r_2}(N_2)$ is transversal to $\bd M_2$ with $\bd \tB_{r_2}(N_2) \cap \bd M_2 \subset \bd \widetilde T_{r_2}(N_2)$.  
	
	\begin{figure}[h]
    	\centering
    	\begin{subfigure}{0.45\linewidth}
    		\centering
        		\def\svgwidth{2.8in}
        		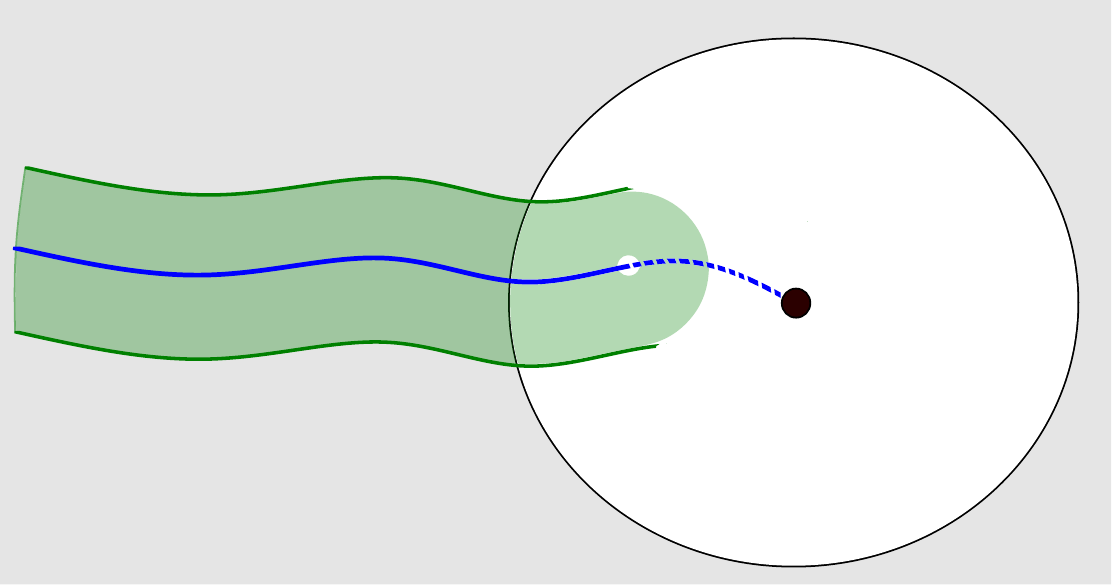
        		\caption{$N_2\subset \tM_{(H_2)}$}\label{Fig: 1}
    	\end{subfigure}
    	\hspace{0.75cm}
    	\begin{subfigure}{0.45\linewidth}
        		\centering
        		\def\svgwidth{2.7in}
        		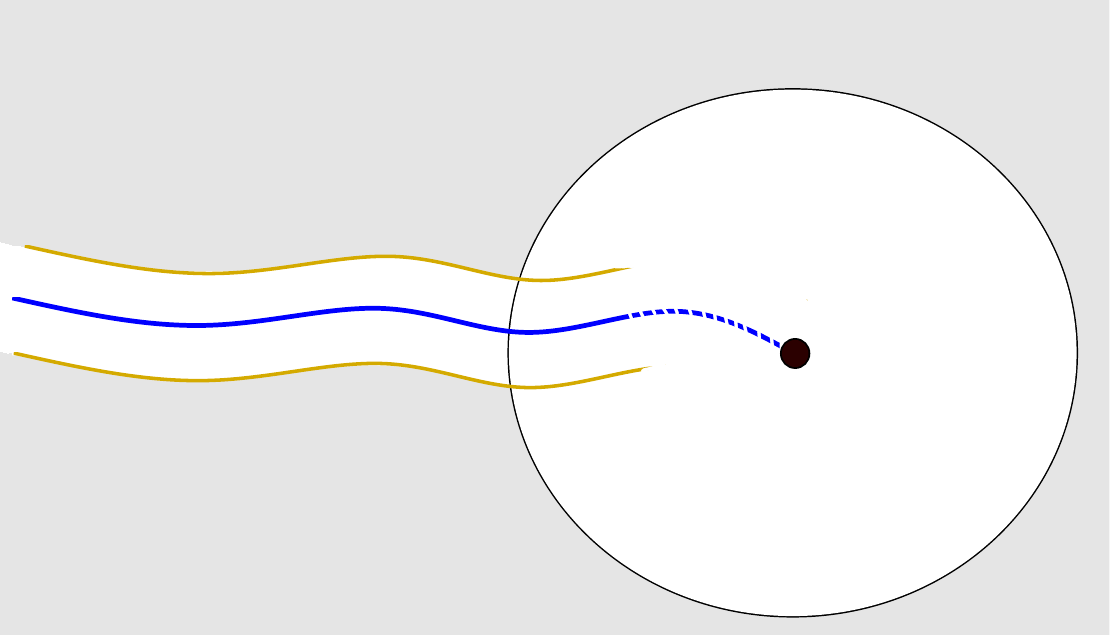
        		\caption{$M_2=M\setminus \tB_{r_2}(N_1)$}\label{Fig: 2}
    	\end{subfigure}
    	\caption{Case (2): $\emptyset\neq \Clos(\tM_{(H_2)})\setminus \tM_{(H_2)} \subset \tM_{(H_1)}$.}
	\end{figure}
	
	Therefore, in both case (1) and (2), we can define 
	\[M_3  := M_2 \setminus \tB_{r_2}(N_2),\]
	which is a compact $G$-domain in $\widetilde{M}$ with piecewise smooth boundary $\partial M_3$ so that 
	\[M_2 \setminus \big( \tB_{t_2}(N_2) \cup \tB_{t_2}(\partial M_2) \big)  ~\subset \subset  ~M_3 ~\subset \subset ~ M_2 \cup \tB_{t_2}(\partial M_2),\]
	and $\tB_{r_2/2}(M_3) \cap (\cup_{i=1}^2 \tM_{(H_i)} ) = \emptyset$.
	
	{\bf Step 4.} {\it Repeat to remove every non-principal orbit type stratum.}
	
	Since the number of non-principal orbit type strata $\{ \widetilde{M}_{(H_i)} \}_{i=1}^I$ is finite, we can repeat this procedure for $I$ times, and obtain $\{t_i \}_{i=0}^I$, $\{r_i \}_{i=0}^I$, $\{F_i\}_{i=0}^I$, $\{N_i\}_{i=0}^I$, and $\{M_i\}_{i=0}^{I+1}$, such that $t_0=r_0>0$, $M_0:=M$, $N_0:=\emptyset$, and
	\begin{itemize}
		\item[(i)] for each $i\in \{1,\dots, I \}$, $\frac{r_{i-1}}{4} > t_i > r_i >0$; 
		\item[(ii)] for each $i\in \{1,\dots, I \}$, 
		$\tM_{(H_i)}$ is the `most singular' orbit type stratum in $\{\tM_{(H_k)}\}_{k=i}^I$ (i.e. none of $k>i$ satisfies $(H_i)\leq (H_k)$) so that either
		\begin{itemize}
			\item[(ii.1)] $\emptyset = \Clos(\tM_{(H_i)})\setminus \tM_{(H_i)}$, and $N_i=\tM_{(H_i)}$ is a closed $G$-submanifold of $\tM$; or
			\item[(ii.2)] $\emptyset\neq  \Clos(\tM_{(H_i)})\setminus \tM_{(H_i)} \subset \cup_{j=1}^{i-1} \tM_{(H_j)}$, and $N_i\subset \widetilde{M}_{(H_{i})}$ is a $G$-domain of $\widetilde{M}_{(H_{i})}$ with smooth boundary satisfying 
		\end{itemize}
		\[\tM_{(H_i)} \cap M_i ~\subset\subset~ N_i ~\subset\subset~ \tM_{(H_i)} \cap \tB_{\frac{r_{i-1}}{2}}(M_i);\]
		\item[(iii)] for each $i\in \{1,\dots, I+1 \}$, $M_i = M_{i-1}\setminus \tB_{r_{i-1}}(N_{i-1})$ is a compact $G$-domain in $\widetilde{M}$ with piecewise smooth boundary so that 
		$M_i \cap \big( \cup_{k=0}^{i-1} N_k  \big) = \emptyset$, 
		\[ M_{i-1} \setminus \big( \tB_{t_{i-1}}(N_{i-1}) \cup \tB_{t_{i-1}}(\partial M_{i-1}) \big) ~ \subset \subset ~ M_i ~ \subset \subset ~ M_{i-1} \cup \tB_{t_{i-1}}(\partial M_{i-1}) ,\]
		and $ \tB_{r_{i-1}/2}(M_i) \cap (\cup_{j=1}^{i-1} \tM_{(H_j)}) = \emptyset$;
		\item[(iv)] for each $i\in \{1,\dots, I \}$, $F_i: \widetilde{M}\to \widetilde{M}$ is a $G$-equivariant map satisfying 
			\begin{eqnarray*}
				M_{i-1} \setminus \big( \tB_{t_{i-1}}(N_{i-1}) \cup \tB_{t_{i-1}}(\partial M_{i-1}) \big) & \subset & F_i\big( M_{i} \setminus \big( \tB_{t_i}(N_i) \cup \tB_{t_i}(\partial M_{i}) \big) \big) 
	 			\\
	 			&\subset & F_i(M_{i+1})
	 			\\
	 			&\subset & F_i\big( M_{i} \cup \tB_{t_i}(\partial M_{i}) \big)  ~\subset~ M_{i-1} \cup \tB_{t_{i-1}}(\partial M_{i-1}),
			\end{eqnarray*}
			$F_i(\Clos(\tB_{t_i}(N_i)))=N_i$, $F_i\llcorner (\tM\setminus\Clos(\tB_{t_i}(N_i)))$ is injective, and ${\rm Lip}(F_i)\leq 1+\epsilon$.
	\end{itemize}
	Clearly, it follows from (ii)(iii) that $M_{I+1}\subset M$ is a compact $G$-domain in $\widetilde{M}^{prin}$ with piecewise smooth boundary. 
	Therefore, we can apply Corollary \ref{Cor: weyl law for single type} to $M_{I+1}$, and see that
	\begin{equation}\label{Eq: wely law for the last cutting manifold}
		\lim_{p\to\infty}p^{-\frac{1}{l}} \omega^G_p(M_{I+1}, g_{_{M}}) = a(l) {\rm Vol}\big(M_{I+1}/G, ~\tilde{g}_{_{M/G}} \big)^{\frac{l-1}{l}}. 
	\end{equation}

	{\bf Step 5.}
	Next, take any $\M$-continuous $(G,p)$-sweepout $\Phi : X\to \Z_{n}^G(M_{I+1},\partial M_{I+1}; \mZ_2)$. 
	Define 
	\[SX := \left\{ (x, \Omega): x\in X, \Phi(x) = [\bd \Omega] \right\} ~ \subset ~ X\times \C^G(M_{I+1}). \]
	Let $\tau: SX\to X$ and $\hat{\Xi}: SX\to \C^G(M_{I+1})$ be the natural projections. 
	Then, one can verify that $\tau: SX\to X$ is a double cover of $X$, and $\hat{\Xi}$ is $\F$-continuous so that $\Phi\circ\tau = [\partial\circ\hat{\Xi}]$ (cf. \cite[Lemma 4.3, Page 954]{liokumovich2018weyl}). 
	Indeed, for any $x\in X$ and $\Phi(x)=[\bd \Omega]$, we have $\tau^{-1}(x)=\{(x,\Omega), (x, M_{I+1}\setminus\Omega)\}$. 

	By (iv) and (c), we have $\big(F_0\circ F_1\circ \cdots \circ F_{I}\big)(M^{I+1}) = M_0=M$. 
	Additionally, since the non-injective domain of $F_i$ will be compressed to a $\mH^{n+1}$-measure zero set, we have
	\[ \Xi := \big(F_0\circ F_1\circ \cdots \circ F_{I}\big) \circ \hat{\Xi} ~:~ SX\to \C^G(M) \]
	is an $\F$-continuous map so that $\Xi(y_1) = M \setminus  \Xi(y_2)$ as two Caccioppoli sets for any $\{y_1, y_2\} = \tau^{-1}(x)$ and $x\in X$.
	Thus, $\Xi$ induces an $\F$-continuous map $\Psi: X\to \Z_{n}^G(M,\partial M; \mZ_2)$ so that $\Psi\circ\tau = [\partial\circ\Xi]$. 
	
	Additionally, since $\M(\bd \Omega) \leq \M(\Phi(x)) + \mH^n(\bd M_{I+1}) $ for any $(x,\Omega)\in SX$, we conclude from ${\rm Lip}(F_i)\leq 1+\epsilon$ that 
	\[\M(\Psi(x)) \leq (1+\epsilon)^{n(I+1)}\M(\Phi(x)) + (1+\epsilon)^{n(I+1)}\mH^n(\partial M_{I+1} ), \qquad {\rm for~all~} x \in X .\]
	Note $\mH^n(\partial M_{I+1} ) \leq  \mH^n(\partial M)+\sum_{i=1}^I \mH^n\big( M \cap \partial \tB_{r_{i}}(N_i) \big)$, and for every $i=1,\dots, I$, $\limsup_{r\to 0}\mH^n\big( M \cap \partial \tB_{r}(N_i) \big) < 2\mH^n\big( M \cap N_{i} \big) + 1$. 
	Hence, we can take a uniform positive constant $C= C(M,G) > \mH^n(\partial M) + 2\mH^n(M\setminus M^{prin}) ) + 1$ so that $\mH^n(\partial M_{I+1} ) \leq C$ for any $\epsilon>0$ and $\{r_i\}_{i=1}^{I}$ sufficiently small. 
	Therefore, 
	\[\sup_{x\in X}\M(\Psi(x)) \leq (1+\epsilon)^{n(I+1)}\sup_{x\in X} \M(\Phi(x)) + (1+\epsilon)^{n(I+1)}C .\]
	
	As in the proof of \cite[Claim in Page955]{liokumovich2018weyl}, $\Psi$ is a $(G,p)$-sweepout with no concentration of mass on orbits. 
	Since $\Phi$ is arbitrary, we have 
	\[ \omega^G_p(M, g_{_{M}}) \leq (1+\epsilon)^{n(I+1)}\omega^G_p(M_{I+1}, g_{_{M}}) + (1+\epsilon)^{n(I+1)}C . \]
	After multiplying the above inequality by $p^{-\frac{1}{l+1}}$, we can use (\ref{Eq: wely law for the last cutting manifold}) to show
	\[\limsup_{p\to\infty} p^{-\frac{1}{l+1}}\omega^G_p(M, g_{_{M}}) \leq (1+\epsilon)^{n(I+1)} a(l){\rm Vol}\big(M_{I+1}/G, \tilde{g}_{_{M/G}} \big)^{\frac{l}{l+1}}.\]
	Finally, by taking $\epsilon\to 0$, (\ref{Eq: upper bound weyl law}) follows from the fact that $M_{I+1}\subset  M^{prin}$. 
\end{proof}

\section{Density of $G$-invariant minimal hypersurfaces}\label{Sec: denseness of FBMHs}

In the final section, we use the Weyl law for the equivariant volume spectrum to show the generic density of $G$-invariant minimal hypersurfaces with free boundary. 

To begin with, let us fix some notations for simplicity. 
Denote by $\Gamma^G$ the set of all $G$-invariant smooth Riemannian metrics on $M$ endowed with the $C^\infty$ topology. 
Then for any $\gamma\in \Gamma^G$, let $\mathcal{N}_G(\gamma)$ be the set of almost properly embedded $G$-invariant FBMHs in $(M,\gamma)$. 
Additionally, define 
\begin{equation}\label{Eq: p-min-max hypersurfaces}
	\mathcal{M}_{G,p}(\gamma):= \left\{\begin{array}{l|l} \Sigma   & \begin{array}{l} \Sigma \mbox{ is a free boundary min-max $((3^p)^{3^p},G)$-hypersurface} \\ \mbox{in $(M,\gamma)$ with multiplicity so that $\|\Sigma\|(M)=\omega^G_p(M,\gamma)$} \end{array}\end{array}\right\},
\end{equation}
which is a {\em non-empty} compact set by Theorem \ref{Thm: compactness theorem}, \ref{Thm: Gp-width realization} provided (\ref{Eq: assumption closed}) or (\ref{Eq: assumption with boundary}). 
Let
\begin{equation}\label{Eq: min-max hypersurfaces}
	\mathcal{M}_{G}(\gamma):= \bigcup_{p\geq 1} 	\mathcal{M}_{G,p}(\gamma).
\end{equation}

Recall that an almost properly embedded FBMH $(\Sigma,\bd\Sigma)$ in $(M,\bd M)$ is said to be {\em degenerate} if there exists $0\neq X\in \mfX^\perp(\Sigma)$ with $Q(X,X)=0$, where $Q$ is defined in (\ref{Eq: second variation formula}). 
Such a vector field is known as a non-trivial {\em Jacobi field} of $\Sigma$. 
Next, we introduce the notations of equivariant bumpy metrics (see \cite{white2017bumpy}\cite{ambrozio2018compactness}\cite{franz2021equivariant}).

\begin{definition}\label{Def: G-bumpy}
	Let $\gamma$ be a $G$-invariant Riemannian metric on M. 
	Then we say $\gamma$ is {\em $G$-bumpy} if no finite cover of any almost properly embedded $G$-invariant FBMH in $M$ admits a non-trivial Jacobi field. 
	
	Denote by $\Gamma^{GB}$ the set of all $G$-bumpy metrics $\gamma\in \Gamma^G$. 
\end{definition}

Note the Jacobi field need not be $G$-invariant in the above definition. 
Moreover, the following proposition indicates the $G$-bumpy metrics are $C^\infty_G$-generic and $G$-invariant FBMHs are $C^\infty_G$-generic countable. 

\begin{proposition}\label{Prop: G-bumpy is dense}
	If ${\rm Cohom}(G)\geq 2$. Then
	\begin{itemize}
		\item[(i)] $\Gamma^{GB}$ is a second-category subset of $\Gamma^G$; 
		\item[(ii)] if $\gamma\in \Gamma^{GB}$ and $M,G$ satisfy either (\ref{Eq: assumption closed}) or (\ref{Eq: assumption with boundary}), then $\mathcal{N}_G(\gamma)$ is countable.
	\end{itemize}
\end{proposition}
\begin{proof}
	For closed $G$-manifolds, (i) and (ii) follow directly from \cite[Theorem 1.3, 3.5]{wang2023equivariant}. 
	
	For compact manifolds with boundary, one can refer to \cite[Theorem 2.8]{guang2021compactness} (which is generalized from \cite[Theorem 9]{ambrozio2018compactness}) for a bumpy metrics theorem of FBMHs. 
	Additionally, \cite[Theorem B.4]{franz2021equivariant} gives the $G$-bumpy metrics theorem for finite group actions on $3$-manifolds, whose arguments are also valid for compact group actions on higher dimensional manifolds as in the closed case \cite[Theorem 1.3]{wang2023equivariant}. 
	Combining these ingredients, we conclude (i). 
	
	To show (ii) in case (\ref{Eq: assumption with boundary}), it is sufficient to show $\{\Sigma\in \mathcal{N}_G(\gamma): \Area(\Sigma)\leq k, {\rm Index}(\Sigma)\leq k\}$ is finite for every $k\in\N$ and $\gamma\in \Gamma^{GB}$. 
	Suppose this is not true.
	Then we have a sequence $\{\Sigma_i\}_{i\in\N}\subset \mathcal{N}_G(\gamma)$ with bounded area and index. 
	By the compactness theorem \cite[Theorem 1.1]{guang2021compactness}, $\Sigma_i$ converges (up to a subsequence) to an almost properly embedded FBMH $\Sigma_\infty\in \mathcal{N}_G(\gamma)$ with multiplicity. 
	Since $\Sigma_i$ and $\gamma$ are both $G$-invariant, we see $\Sigma_\infty$ is also $G$-invariant. 
	Finally, by \cite[Theorem 1.2]{guang2021compactness} and \cite[Theorem 1.1]{wang2019compactness}, there is a non-trivial Jacobi field on $\Sigma_\infty$ or its double cover, which contradicts the choice of $\gamma\in \Gamma^{GB}$. 
\end{proof}

\subsection{Density of $G$-invariant FBMHs}
We now combine the arguments in \cite{irie2018density}\cite{li2023existence} and \cite{wang2022existence} to show the existence of infinitely many $G$-invariant FBMHs. 

 Given any non-empty (resp. relative) open $G$-set $U\subset \bd M$ (resp. $U\subset M$), we define 
 \begin{equation}\label{Eq: density metrics p}
 	\Gamma^G_{U,p} := \left\{ \gamma\in \Gamma^G: ~\forall \Sigma\in \mathcal{M}_{G,p}(\gamma), ~\bd\Sigma\cap U \neq \emptyset  \mbox{ (resp. $\Sigma\cap U\neq \emptyset$)}\right\},
 \end{equation}
as a subset of $\Gamma^G$, where $p\in\mZ_+$. 
In addition, let
\begin{equation}\label{Eq: density metrics}
 	\Gamma^G_{U} := \bigcup_{p\geq 1} \Gamma^G_{U,p}. 
\end{equation}

\begin{theorem}\label{Thm: open dense of metrics}
	Suppose $M$ and $G$ satisfy either (\ref{Eq: assumption closed}) or (\ref{Eq: assumption with boundary}).  
	Then for any nonempty open (resp. relative open) $G$-set $U \subset \bd M$ (resp. $U \subset M$), $\Gamma^G_U$ is an open dense subset of $ \Gamma^G$ in the $C^\infty$ topology. 
\end{theorem}
\begin{proof}
	Firstly, we claim $\Gamma^G_{U,p}$ is open for each $p\in\mZ_+$. 
	Otherwise, there exists $\gamma_0\in \Gamma^G_{U,p}$ and a sequence $\{\gamma_i\}_{i\in\N}\subset \Gamma^G\setminus \Gamma^G_{U,p}$ so that $\gamma_i\to\gamma_0$ in the smooth topology. 
	By definitions, we can take $\Sigma_i\in \mathcal{M}_{G,p}(\gamma_i)$ so that $\bd\Sigma_i\cap U=\emptyset$ (resp. $\Sigma_i\cap U =\emptyset$). 
	Note $\omega^G_p(M,\gamma_i)\to \omega^G_p(M,\gamma_0)$ by Lemma \ref{Lem: Gp-width depends continuously on metric}. 
	It then follows from the compactness theorem \ref{Thm: compactness theorem} that $\Sigma_i$ converges (up to a subsequence) smoothly away from a finite union of orbits $\mathcal{Y}=\{G\cdot p_k\}_{k=1}^K$ to some $\Sigma_0\in \mathcal{M}_{G,p}(\gamma_0)$. 
	Since $\mH^{n-1}(\mathcal{Y}) = 0$ by (\ref{Eq: assumption closed})(\ref{Eq: assumption with boundary}), we see from the smooth convergence that $\bd\Sigma_0\cap U=\emptyset$ (resp. $\Sigma_0\cap U =\emptyset$), which contradicts the choice of $\gamma_0\in \Gamma^G_{U,p}$. 
	
	Hence, $\Gamma^G_{U,p}$ (for all $p\in\mZ_+$) and $\Gamma^G_{U}$ are open in $\Gamma^G$. 
	Next, we show $\Gamma^G_{U}$ is dense in $\Gamma^G$. 
	
	Assume by contradiction that there exists an open set $\B$ in $\Gamma^G$ so that $\Gamma^G_U \cap \B = \emptyset$. 
	Then, it follows from the definitions (\ref{Eq: density metrics p})(\ref{Eq: density metrics}) that 
	\begin{equation}\label{Eq: all in M-U}
		\mbox{$\forall g_{_{M}}\in \B $, $p\in\mZ_+$},~\mbox{ $\exists \Sigma\in \mathcal{M}_{G,p}(g_{_M})$ }~ \mbox{ s.t. }~ \bd\Sigma\subset \bd M\setminus U ~({\rm resp.}~\Sigma\subset M\setminus U).
	\end{equation}
	Additionally, by Proposition \ref{Prop: G-bumpy is dense}(i)(ii), we can always find a $G$-bumpy metric $g_{_M}^0\in \B \cap \Gamma^{GB}$ so that $\mathcal{N}_G(g_{_M}^0)$ is a countable set. 
	Thus, we have the following countable set 
	\[ \mathcal{C} := \left\{ \sum_{i=1}^I m_i {\rm Area}_{g_{_M}^0}(\Sigma_i) : I\in\N, ~\{m_i\}_{i=1}^I\subset\N, ~ \{\Sigma_i\}_{i=1}^I\subset \mathcal{N}_G(g_{_M}^0)  \right\}. \]
	
	{\bf Case I:} {\em $U$ is a $G$-invariant open set in $\bd M$.} 
	
	In the $G$-manifold $\bd M$, we can take a principal orbit $G\cdot x \subset U$ by the density of $(\bd M)^{prin}$. 
	Since the inward unit normal of $M$ along $\bd M$ is $G$-invariant, we must have $G\cdot x \subset M^{prin}$ by \cite[Corollary 2.2.2]{berndt2016submanifolds}. 
	Now, we take a relative open $G$-neighborhood $K\subset M^{prin}$ of $G\cdot x$ with $K\cap\bd M\subset U $. 
	Let $X$ be a $G$-invariant vector field compactly supported in $K$ so that $X$ points inward $M$ everywhere on $\spt(X)\cap\bd M \neq \emptyset$. 
	Hence, the equivariant diffeomorphisms $\{f_t\}_{t\in [0,\delta)}$ generated by $X$ will push $M$ into a $G$-invariant sub-domain $f_t(M) \subset \subset M$ for $t\in (0,\delta)$. 
	By shrinking $\delta>0$, we can make the pull-back metric $g_{_M}^t := f_t^*g_{_M}^0 \in \B$ for all $t\in [0,\delta)$. 
	Therefore, 
	\begin{equation}\label{Eq: density: isometry}
		f_t: \left(M,g_{_M}^t \right)\to \left(f_t(M), g_{_M}^0\right)\subset \left(M,g_{_M}^0\right)
	\end{equation}
	is an isometry for all $t\in [0,\delta)$.

	Next, given any $t\in [0,\delta)$ and $p\in \mZ_+$, it follows from (\ref{Eq: all in M-U}) that there exists $\Sigma_p^t\in \mathcal{M}_{G,p}(g_{_M}^t)$ with $\bd \Sigma_p^t \subset \bd M\setminus U$. 
	Combining with (\ref{Eq: density: isometry}), we see $ f_t(\Sigma_p^t)$ is a free boundary min-max $((3^p)^{3^p},G)$-hypersurface in $(f_t(M), g_{_M}^0)$ with multiplicity so that
	\begin{equation}\label{Eq: density: boundary not in U}
		\bd f_t(\Sigma_p^t) \subset \bd f_t(M) \setminus f_t(U). 
	\end{equation}
	Additionally, noting $(f_t(M), g_{_M}^0)\subset (M,g_{_M}^0)$ is a $G$-invariant sub-domain, (\ref{Eq: density: boundary not in U}) also implies $\spt(f_t(\Sigma_p^t))\in \mathcal{N}_G(g^0_{_M})$, and thus
	\[ \omega^G_p(M,g^t_{_M}) = \|\Sigma_p^t\|_{g^t_{_M}}(M) = \|f_t(\Sigma_p^t)\|_{g_{_M}^0}(f_t(M)) =\|f_t(\Sigma_p^t)\|_{g_{_M}^0}(M) \in \C . \]
	Since $\mathcal{C}$ is countable and $\omega^G_p(M, g_{_M}^t)$ is continuous in $t$ (Lemma \ref{Lem: Gp-width depends continuously on metric}), we conclude that $\omega^G_p(M, g_{_M}^t) \equiv \omega^G_p(M, g_{_M}^0) $ for all $t\in [0,\delta)$, $p\in\mZ^+$. 
	This contradicts Theorem \ref{Thm: main theorem weyl law} because ${\rm Vol}(M^{prin}/G, \tilde{g}_{_{M/G}}^t) = {\rm Vol}((f_t(M))^{prin}/G, \tilde{g}_{_{M/G}}^0)$ is strictly decreasing. 

	{\bf Case II:} {\em $U$ is a $G$-invariant relative open set in $M$.} 
	
	Since $M^{prin}$ is dense in $M$, we can take a $G$-invariant smooth nonnegative function $h: M\to [0,\infty)$ so that $\spt(h)=K\subset\subset U $ and $h(x)>0$ for some $x\in U\cap M^{prin}$. 
	Consider the continuous curve $g_{_M}^t := (1+th)\cdot g_{_M}\in \B$, where $t\in [0,\delta)$ for some $\delta>0$ small enough. 
	
	Then for any $t\in [0, \delta)$ and $p\in\mZ_+$, let $\Sigma_p^t\in \mathcal{M}_{G,p}(g_{_M}^t)$ be given by (\ref{Eq: all in M-U}) so that $\Sigma_p^t\subset M\setminus U$. 
	Note $g_{_M}^t = g_{_M}^0$ in $M\setminus K$ for all $t\in [0,\delta)$. 
	Thus, we have $\spt(\Sigma_p^t)\in \mathcal{N}_G(g^0_{_M})$, which similarly implies $\{\omega^G_p(M, g_{_M}^t)\}_{p=1}^\infty \subset \mathcal{C}$ for all $t\in [0,\delta)$. 
	As in {\bf Case I}, since $\mathcal{C}$ is countable and $\omega^G_p(M, g_{_M}^t)$ is continuous in $t$ (Lemma \ref{Lem: Gp-width depends continuously on metric}), we see $\omega^G_p(M, g_{_M}^t) \equiv \omega^G_p(M, g_{_M}^0) $ for all $t\in [0,\delta)$ and $p\in\mZ^+$, which contradicts Theorem \ref{Thm: main theorem weyl law} as ${\rm Vol}(M^{prin}/G, \tilde{g}_{_{M/G}}^t)$ is strictly increasing. 
	
	Therefore, we have $\Gamma^G_U$ is dense in $\Gamma^G$ in both cases. 
\end{proof}

Finally, we show the proof of our main theorems (Theorem \ref{Thm: main density}, \ref{Thm: main infinitely many singular MHs}). 

\begin{proof}[Proof of Theorem \ref{Thm: main density}]
	Let $\{U_i\}_{i\in\N}$ be a countable basis of $\bd M$ (resp. $M$). 
	Since the quotient map $\pi$ is an open map, $\{\pi(U_i)\}_{i\in\N}$ is also a countable basis of $\bd M/G$ (resp. $M/G$). 
	Define $\widetilde{U}_i := G\cdot U_i = \pi^{-1}(\pi(U_i))$ for every $i\in\N$. 
	We conclude from Theorem \ref{Thm: open dense of metrics} that each $ \Gamma^G_{\widetilde{ U}_i}$, $i=1,2,\dots$, is open and dense in $\Gamma^G$. 
	Hence, the set $\cap_{i\in\N} \Gamma^G_{\widetilde{ U}_i}$ is $C^\infty$ Baire-generic (second-category) in $\Gamma^G$, which gives Theorem \ref{Thm: main density} by the definitions of $\widetilde{U}_i$ and $\Gamma^G_{\widetilde{ U}_i}$. 
\end{proof}

\begin{proof}[Proof of Theorem \ref{Thm: main infinitely many singular MHs}]
	Redefine $\mathcal{N}_G(\gamma)$ to be the set of closed $G$-invariant minimal hypersurfaces in $(M,\gamma)$ with a singular set of Hausdorff dimension no more than $n-7$. 
	As we mentioned in Remark \ref{Rem: singular MHs}, Theorem \ref{Thm: compactness theorem} and \ref{Thm: Gp-width realization} also hold for these closed min-max $G$-hypersurfaces with $7$-codimsional singularity by the regularity results in \cite{schoen1981regularity}. 
	Consider 
	\[ \Gamma^G_0 := \left\{\gamma\in\Gamma^G :~ \mathcal{N}_G(\gamma) \mbox{ is finite} \right\},\]
	which is dense in $\interior(\overline{\Gamma^G_0})$. 
	Then, the proof of Theorem \ref{Thm: open dense of metrics} would carry over to show $\Gamma^G_U$ is open and dense in $\interior(\overline{\Gamma^G_0})$ by using $\Gamma^G_0, \interior(\overline{\Gamma^G_0})$ in place of $\Gamma^{GB},\Gamma^G$.  
	Let $\{U_i\}_{i\in\N}$ be a countable basis of $M$, and $\widetilde{U}_i := G\cdot U_i = \pi^{-1}(\pi(U_i))$. 
	Then we have $\cap_{i\in\N} \Gamma^G_{\widetilde{ U}_i}$ is $C^\infty_G$ Baire-generic (second-category) in $\interior(\overline{\Gamma^G_0})$, which implies $\Gamma^G_0\cap  \interior(\overline{\Gamma^G_0}) \subset \interior(\overline{\Gamma^G_0}) \setminus (\cap_{i\in\N} \Gamma^G_{\widetilde{ U}_i})$ is of the first category in $\Gamma^G$. 
	Finally, since $\bd (\overline{\Gamma^G_0}) $ is nowhere dense, we have $\Gamma^G_0$ is of the first category in $\Gamma^G$. 
\end{proof}

\bibliographystyle{abbrv}

\providecommand{\bysame}{\leavevmode\hbox to3em{\hrulefill}\thinspace}
\providecommand{\MR}{\relax\ifhmode\unskip\space\fi MR }
\providecommand{\MRhref}[2]{%
  \href{http://www.ams.org/mathscinet-getitem?mr=#1}{#2}}
\providecommand{\href}[2]{#2}

\bibliography{reference}   

\end{document}